\documentclass[11pt,a4paper]{article}

\usepackage[a4paper,margin=30mm]{geometry}
\usepackage[T1]{fontenc}
\usepackage[utf8]{inputenc}
\usepackage{mathpazo}
\usepackage{microtype}
\usepackage{amsmath,amssymb,amsthm,amscd,mathtools}
\usepackage{mathrsfs}
\usepackage{graphicx}
\usepackage{tikz}
\usetikzlibrary{arrows.meta,positioning,calc}
\usepackage{enumerate}
\usepackage{cite}
\usepackage{float}
\usepackage{tocloft}
\usepackage{xcolor}
\usepackage[all]{xy}
\usepackage{hyperref}
\hypersetup{
  colorlinks=true,
  linkcolor=blue!55!black,
  citecolor=blue!55!black,
  urlcolor=blue!55!black,
  pdftitle={Boundary-weighted barycenter spaces},
  pdfauthor={Mohameden Ahmedou, Sadok Kallel},
  pdfsubject={Topology of boundary-weighted barycenter spaces and an application to the resonant Neumann mean-field equation},
  pdfkeywords={boundary-weighted barycenters, formal barycenters, symmetric joins, Morse theory at infinity}
}
\allowdisplaybreaks[4]
\numberwithin{equation}{section}

\newtheorem{thm}{Theorem}[section]
\newtheorem{lem}[thm]{Lemma}
\newtheorem{prop}[thm]{Proposition}
\newtheorem{cor}[thm]{Corollary}
\theoremstyle{definition}
\newtheorem{defn}[thm]{Definition}
\newtheorem{exa}[thm]{Example}
\theoremstyle{remark}
\newtheorem{rem}[thm]{Remark}

\newcommand{\R}{\mathbb{R}}
\newcommand{\Z}{\mathbb{Z}}
\newcommand{\N}{\mathbb{N}}
\newcommand{\F}{\mathbb{F}}
\newcommand{\Q}{\mathbb{Q}}
\newcommand{\Si}{\Sigma}
\newcommand{\pa}{\partial}
\newcommand{\ch}{\chi}
\newcommand{\wt}{\widetilde}

\newcommand{\colim}{\operatorname*{colim}}
\newcommand{\supp}{\operatorname{supp}}
\newcommand{\BB}{\mathcal B}

\newcommand{\SP}{\operatorname{SP}}
\newcommand{\conn}{\operatorname{conn}}
\newcommand{\mfg}{\mathfrak g}

\title{\bfseries Boundary-weighted barycenter spaces}
\author{
Mohameden Ahmedou\thanks{Mathematisches Institut, Justus-Liebig-Universit\"at Giessen,
Arndtstrasse 2, 35392 Giessen, Germany. M. Ahmedou was supported in part by
DFG grant AH 156/2-1.}
\and
Sadok Kallel\thanks{American University of Sharjah, United Arab Emirates, and
Laboratoire Paul Painlev\'e, Universit\'e de Lille, France.}
}

\begin{document}
\maketitle

\begin{abstract}
We study boundary-weighted barycenter spaces, in which an interior support point has cost (or weight) two and a boundary support point has cost one.  These spaces are finite-dimensional topological models for the concentration patterns produced by noncompact boundary Euler--Lagrange functionals: an interior bubble carries twice the quantized mass of a boundary bubble, and very negative sublevels are therefore modeled by boundary-weighted rather than ordinary barycenters.  The paper gives a systematic algebraic-topological treatment of these spaces.  We construct the mixed strata, the closed-stratum poset, and the triangular boundary-weighted colimit filtration; compute the Euler characteristic; and prove a homology decomposition in terms of ordinary barycenter spaces of \(\partial M\) and \(M/\partial M\).  We then specialize the formulas to compact connected orientable surfaces with boundary and to hemispheres, obtaining explicit rational Betti polynomials and, for surfaces, the mod-two polynomials required by the mean-field application.  Finally, we explain how these polynomials measure the topology at infinity in the resonant Neumann mean-field equation on a compact surface with boundary.
\end{abstract}

\noindent\textbf{Keywords.} Boundary-weighted barycenters; formal barycenters; symmetric joins; colimits; Betti numbers; Euler characteristic; hemispheres; Morse inequalities at infinity.

\smallskip
\noindent\textbf{2020 Mathematics Subject Classification.} 55N10, 55R80, 58E05, 35J60.

\tableofcontents

\section{Introduction and main results}
\label{sec:intro}

The purpose of this paper is to isolate a topological object which appears naturally in boundary concentration problems and to compute it in a form usable in Morse theory at infinity.  The object is the \emph{boundary-weighted barycenter space} of a compact manifold with boundary.  Its defining rule is simple: a point in the interior of the manifold has weight, or cost, two, while a point on the boundary has cost one.  Thus a mixed configuration with \(p\) interior points and \(q\) boundary points belongs to the level \(l\) precisely when
\[
2p+q\le l.
\]
This convention is not artificial.  It is the topological form of the mass quantization in boundary elliptic problems.  For the Neumann mean-field equation on a surface, an interior bubble carries mass \(8\pi\), whereas a boundary bubble carries mass \(4\pi\).  In the boundary \(Q\)-curvature problem the analogous fourth-order quantization has the same doubling principle.  Boundary-weighted barycenters are therefore the spaces in which the limiting concentration profiles of such equations live.

Our basic building block is the ordinary barycenter space
\begin{equation}\label{baryspace}
\BB_n(X)=\left\{\sum_{i=1}^n t_i\delta_{x_i}: x_i\in X,
\ t_i\ge0,\ \sum_i t_i=1\right\}
\end{equation}
which is figuratively the union of all simplices spanned by at most \(n\) points of \(X\), with the usual quotient identifications obtained by deleting zero weights and merging coincident support points (see \S\ref{prelim}). On compact metric spaces, this is also identified with a subspace of
 the space of \textit{finitely supported probability measures} under the Lévy--Prokhorov topology \cite{DJ, KW}. Kallel and Karoui identified this construction with symmetric joins and, after one suspension, with reduced symmetric products \cite{KallelKaroui2011}.  Hence ordinary barycenters are governed by the same algebra as symmetric powers: Dold--Thom theory, symmetric products, joins, and functorial homology.
 
 The spaces studied in this paper are a variant of the spaces \eqref{baryspace} defined on a compact connected smooth manifold $M$ of dimension at least two with nonempty boundary. The space denoted $\BB_l^\pa (M)$ consists of formal probability measures on \(M\) whose effective support satisfies the ``cost constraint'' \(2p+q\le l\), where \(p\) is the number of interior support points and \(q\) the number of boundary support points.
 This is by construction a subspace of $\BB_{l}(M)$, and it is naturally filtered
\begin{equation}\label{filtration}
\BB_1^{\pa}(M) = \partial M\subset \BB_2^{\pa}(M)\subset\cdots \subset \BB_l^\pa (M)
\end{equation}

These variants are introduced for two reasons. First, they describe an actual topological jump.  The filtration \eqref{filtration}
is not merely a convenient indexing device.  It records how new concentration patterns appear when the allowed total quantized mass is increased.  
Secondly, boundary-weighted barycenters are the finite-dimensional models for very low sublevels of noncompact Euler--Lagrange functionals.  In variational problems with critical exponential growth, improved Moser--Trudinger or Adams inequalities imply that a function with very negative energy must concentrate its normalized density near a barycentric configuration.  Conversely, bubble test maps send such configurations into very negative sublevels.  At a resonant parameter \(\kappa\ge2\), this produces, in the Neumann mean-field problem treated in Section \ref{sec:variational-jump}, a homotopy equivalence of the form
\[
J^{-L}\simeq \BB_{\kappa-1}^{\pa}(M),\qquad L\gg1,
\]
On the other hand, large positive sublevels are contractible.  Consequently the relative topology of the variational pair is
\[
H_r(J^L,J^{-L};\F)\cong \widetilde H_{r-1}(\BB_{\kappa-1}^{\pa}(M);\F),
\]
and the shifted reduced Poincar\'e polynomial of \(\BB_{\kappa-1}^{\pa}(M)\) is exactly the topological term entering the strong Morse relations.
This is the point at which topology meets analysis. 

In a compact variational problem, the change in sublevel topology is accounted for by genuine critical points.  In the boundary problems motivating this paper, the Palais--Smale condition fails at resonance.  Noncompact negative-gradient orbits may escape by bubbling, and their limiting profiles are critical points at infinity in the sense of Bahri \cite{Bahri1989,BahriCoron1988,BahriCoron1991}.  The barycenter space measures the homological demand imposed by the bottom sublevel, while the genuine critical points and the critical points at infinity provide the Morse generators.  Thus the algebraic-topological computation of \(\BB_l^{\pa}(M)\) is precisely what is needed to write the topology jump at infinity as an explicit Morse relation.
The same mechanism occurs in several geometric and analytic problems.  In the resonant boundary \(Q\)-curvature problem of Ahmedou--Kallel--Ndiaye, the very negative sublevels of the Paneitz--Chang--Qing functional are modeled by boundary-weighted barycenters, and the corresponding critical points at infinity contribute to the Morse theory of the problem \cite{AKN2016}.  Weighted barycenter spaces also occur in singular Liouville equations and in Toda systems, where the weights encode the allowed concentration masses of singular sources or of several components \cite{CarlottoMalchiodi2012,malchiodi,JKM2015,Kallel2019}.  

This paper can be used both as a standalone computation of boundary-weighted barycenters and as a reference for PDE applications in which these spaces provide the required barycentric topology at infinity. 
Our first result on connectivity and Euler characteristic $\chi$ is obtained by combining Corollary \ref{prop:connectivity} and Lemma \ref{lem:Bpq-euler} from the text.

\begin{thm}\label{thm:intro-euler}
Assume that both \(M\) and \(\pa M\) are \(r\)-connected, $r\geq 1$. Then 
\begin{equation}\label{eq:connectivity-even}
\BB_l^{\pa}(M)\text{ is at least}\ \begin{cases}(l+r-2)\text{-connected}& \hbox{if $l$ even}\\
(l+r-1)\text{-connected}& \hbox{if $l$ odd}
\end{cases}
\end{equation} 
The space $\BB_l^\pa (M)$ is simply connected as soon as $l\ge 2$.

Moreover, when $M$ is even-dimensional with nonempty boundary and \(l\ge1\),
\begin{equation}\label{eq:intro-euler}
\ch(\BB_{2l}^{\pa}(M))=\ch(\BB_l(M))
\ \hbox{and}\ 
\ch(\BB_{2l-1}^{\pa}(M))=\ch(\BB_{l-1}(M)).
\end{equation}
In particular, if \(\ch(M)=0\), then \(\ch(\BB_l^{\pa}(M))=0\) for every \(l\ge1\).
\end{thm}

Our second main result consists of a general homology decomposition of the boundary barycenter spaces. 
All homology groups are taken with coefficients in a field \(\F\).  We write \(\wt H_*(M)\) for the reduced homology of $M$ (these are graded vector spaces with $\wt H_0(M)=0$ since $M$ is connected). We assume \(\wt H_*(\varnothing;\F)=0\).  

\begin{thm}\label{thm:intro-general-homology}
Let \(M\) be a compact connected smooth manifold of dimension at least two with nonempty boundary.  For every \(l\ge1\), with field coefficients,
\begin{equation}\label{eq:intro-general-even}
\begin{aligned}
\widetilde H_k\bigl(\BB_{2l}^{\partial}(M);\mathbb F\bigr)
\cong {}&
\widetilde H_k\bigl(\BB_{2l}(\partial M);\mathbb F\bigr)
\oplus
\widetilde H_k\bigl(\BB_l(M/\partial M);\mathbb F\bigr)
\\
&\oplus
\bigoplus_{j=1}^{l-1}
\bigoplus_{\substack{a,b\ge0\\a+b=k-1}}
\widetilde H_a\bigl(\BB_j(M/\partial M);\mathbb F\bigr)
\otimes_{\mathbb F}
\widetilde H_b\bigl(\BB_{2l-2j}(\partial M);\mathbb F\bigr).
\end{aligned}
\end{equation}
\begin{equation}\label{eq:intro-general-odd}
\begin{aligned}
\widetilde H_k\bigl(\BB_{2l-1}^{\partial}(M);\mathbb F\bigr)
\cong {}&
\widetilde H_k\bigl(\BB_{2l-1}(\partial M);\mathbb F\bigr)
\\
&\oplus
\bigoplus_{i=1}^{l-1}
\bigoplus_{\substack{a,b\ge0\\a+b=k-1}}
\widetilde H_a\bigl(\BB_i(M/\partial M);\mathbb F\bigr)
\otimes_{\mathbb F}
\widetilde H_b\bigl(\BB_{2l-2i-1}(\partial M);\mathbb F\bigr).
\end{aligned}
\end{equation}
\end{thm}

As is apparent, the homology of \(\BB_l^{\pa}(M)\) decomposes into ordinary barycenter homology of \(\pa M\), barycenter homology of the quotient \(M/\pa M\), and join terms.  Thus the boundary contribution, the interior contribution, and the genuinely mixed contribution are separated algebraically.  
The rank-filtration argument suggests a homotopy-level wedge decomposition of the corresponding spaces; however, the proof below establishes only the stated homology decomposition, because wedge cancellation is unavailable in general (see \eqref{splitformula}).

The homology groups in the decompositions appearing in Theorem \ref{thm:intro-general-homology} can be computed in terms of the homology of reduced symmetric products by work of \cite{KallelKaroui2011}. The most economical general formulas are rational; these are derived in Sections \ref{sec:hemispheres} and \ref{sec:surface-specialization} for disks/hemispheres and for orientable surfaces with boundary. Section \ref{sec:surface-specialization} also records the mod-two specialization used in Section \ref{sec:variational-jump}.

We write the reduced Poincar\'e polynomial with field coefficients $\F$ as
$$\displaystyle 
\wt P_X(t;\F):=\sum_{r\ge0}\dim_{\F}\wt H_r(X;\F)t^r.$$

\begin{cor}\label{closeddisks} Let \(M\) be the closed $d$-dimensional disk, equivalently the hemisphere $H^d$, and assume $d\ge2$ is even. Then
with rational coefficients, we have the following Poincar\'e polynomials
\begin{align*}
\wt P_{\BB_{2l}^{\pa}(M)}(t;\Q)
&=
\begin{cases}
t^{2dl-1}+t^{2dl-d},&l\ge2\\
t^{2d-1}+t^{d},&l=1
\end{cases}\\
\wt P_{\BB_{2l-1}^{\pa}(M)}(t;\Q)
&=
\begin{cases}
t^{d(2l-1)-1}+t^{d(2l-2)},&l\ge2\\
t^{d-1},&l=1
\end{cases}
\end{align*}
Simpler formulas hold for odd dimensional disks, hemispheres or holed spheres (Proposition \ref{odddimensionalcase}).
\end{cor}

\begin{rem}\rm Note that the Euler characteristics obtained from these polynomials agree with Theorem \ref{thm:intro-euler}: Since $H^d$ is contractible, so are $\BB_l(H^d)$ for all $l\ge 1$, with $\chi (\BB_l(H^d))=1$. On the other hand \(\ch(\BB_{2l}^{\pa}(H^{d}))=1\) for every \(l\ge1\), while \(\ch(\BB_{2l-1}^{\pa}(H^{d}))=0\) for \(l=1\) and equals \(1\) for \(l\ge2\).
\end{rem}

The next application is to Riemann surfaces with boundary, and the computations are more tedious. The following is the content of Lemma \ref{QNn} and Proposition \ref{prop:Rbn-formula}. 

\begin{cor}\label{surface-case} Let \(M=\Si\) be a compact connected orientable surface of genus \(\mathfrak g\) with boundary
$
\pa\Si=\Gamma_1\sqcup\cdots\sqcup\Gamma_b$, where each connected component $\Gamma_j$ is a copy of the circle $S^1$. Let $N=2\mfg + b-1$ and assume $N\geq 1$.
Then
\begin{align*}
\wt P_{\BB_{2l}^{\pa}(\Si)}(t;\Q)
&=
R_{b,2l}(t)+Q_{N,l}(t)
+t\sum_{i=1}^{l-1}
Q_{N,i}(t)R_{b,2l-2i}(t),
\\
\hbox{and}\ \ \wt P_{\BB_{2l-1}^{\pa}(\Si)}(t;\Q)
&=
R_{b,2l-1}(t)
+t\sum_{i=1}^{l-1}
Q_{N,i}(t)R_{b,2l-2i-1}(t)
\end{align*}
where
\begin{eqnarray*}
Q_{N,n}(t)
&=&
\binom{N+n-1}{n}t^{2n-1}
+
\binom{N+n-2}{n-1}t^{2n}\\
R_{b,n}(t)&=&
\sum_{c=0}^{\min\{b-1,n\}}
\binom{b-1}{c}\binom{b+n-c-1}{n-c}t^{2n-c-1}.
\end{eqnarray*}
\end{cor}

In the case $M=\Sigma$, the rational surface formulas of Corollary \ref{surface-case} evaluated at \(t=-1\) reproduce exactly Theorem \ref{thm:intro-euler}.
Similar computations are obtained over the field $\mathbb F_2$ (Lemma \ref{lem:Q-mod-two}), and then used in our final section \S\ref{sec:variational-jump}.

The paper is organized as follows.  Sections \ref{sec:def}--\ref{sec:general-homology} develop the closed strata, the triangular colimit filtration, the Euler-characteristic calculation, and the homology decomposition.  Sections \ref{sec:hemispheres} and \ref{sec:surface-specialization} specialize the formulas to hemispheres and to compact orientable surfaces with boundary. Section \ref{sec:variational-jump}
 focuses on the two-dimensional Neumann mean-field equation as in Ahmedou--Hu--Wang \cite{AHW2026} and explains how the mod-two polynomials enter the topology at infinity in the resonant case: at resonance \(\rho=4\pi\kappa\), \(\kappa\ge2\), a mixed critical point at infinity has type \(2p+q=\kappa\), and the Betti numbers computed here give the right-hand side of the resonant strong Morse relation.\\
 
 \vskip 5pt
\noindent{\sc Acknowledgements}: The authors used ChatGPT to assist with the algebraic manipulations in the generating-series calculations of \S\ref{sec:surface-specialization} leading to Corollary \ref{surface-case}. 
 \vskip 5pt

\section{Barycenter spaces}\label{prelim}

In this paper, $\mathfrak S_k$ will denote the symmetric group on $k$ letters. As in \cite{KallelKaroui2011, KW}, we define the barycenter spaces \eqref{baryspace} as a quotient of the product
\[
\BB_n(X)=\bigsqcup_{k=1}^n X^k\times \Delta^{k-1}\slash_{\sim}
\] where 
$\displaystyle\Delta^{k-1}:=\left\{(t_1,\ldots, t_k)\in[0,1]^k:\sum_{i=1}^kt_i=1\right\}$
is the $k-1$-dimensional simplex, and
the equivalence relation $\sim$ is generated by the following three relations:
\begin{enumerate}\setlength{\itemsep}{0pt}
\item $(x_1,\ldots, x_k,t_1,\ldots, t_k)\sim(x_{\sigma(1)},\ldots, x_{\sigma(k)}, t_{\sigma(1)},\ldots, t_{\sigma(k)})$ for $\sigma\in \mathfrak S_k$;
\item $(x_1,\ldots, x_k,t_1,\ldots, t_k)\sim (x_1,\ldots, x_{k-1}, t_1,\ldots, t_{k-1}+t_k)$ if $x_{k-1}=x_k$; and
\item $(x_1,\ldots, x_k,t_1,\ldots, t_k)\sim (x_1,\ldots, x_{k-1}, t_1,\ldots, t_{k-1})$ if $t_k=0$.
\end{enumerate}
This space $\BB_n(X)$ is a closed subspace of every $\BB_m(X)$, with $m>n$, and is in fact a closed subspace of the colimit $\BB(X):= \BB_\infty(X)$. 

The space $\BB(X)$ is also known as the space of \textit{finitely supported probability measures}.
More precisely,  a probability measure with finite support is a convex linear combination $\sum_{i=1}^nt_i\delta_{x_i}$, where $\delta_x$ is the point mass or Dirac delta ``function'' assigning full measure to the singleton $\{x\}$. 
Every point of $\BB(X)$ can be written uniquely in the \textit{reduced} form $\sum_{i=1}^nt_i\delta_{x_i}$ subject to the following two conditions: 
$$\hbox{(i) $t_i>0$ for $1\leq i\leq n$; and (ii) $x_i\neq x_j$ for $i\neq j$}$$
When written in this reduced form, the coefficient of $\delta_x$ in $\mu\in \BB(X)$ is well-defined and denoted $\mu(x)$. As defined in \cite{DJ, KW}, the \emph{support} of $\mu\in\BB(X)$ is the (finite) subset $\mathrm{supp}(\mu)=\{x\in X: \mu(x)\neq0\}$. Under this definition, $\BB_n(X)\subset \BB(X)$ coincides with the subspace of probability measures $\mu$ with $|\mathrm{supp}(\mu)|\leq n$.  This is a closed subspace under the L\'{e}vy--Prokhorov
topology that we discuss next.

The Lévy--Prokhorov metric $\rho$ is defined on $\BB(X)$ for metrizable $X$. The distance between any two given probability measures $\mu,\nu\in \BB(X)$ is given by
$$
\rho_X(\mu,\nu)
=
\inf\left\{
\varepsilon>0:
\begin{array}{l}
\mu(A)\leq \nu(A^\varepsilon)+\varepsilon,\\[1mm]
\nu(A)\leq \mu(A^\varepsilon)+\varepsilon,
\end{array}
\ \text{for every Borel set }A\subset X
\right\},
$$
where \(A\) ranges over the Borel \(\sigma\)-algebra of \(X\), and $A^{\varepsilon}$ is the $\varepsilon$-neighborhood of $A$ defined as 
$$A^{\varepsilon}=\{x\in X \mid d(x,y) < \varepsilon \text{ for some } y\in A\}.$$   
Here, the infimum is taken over all measurable sets $A$. 
A basis at a point $\mu=\sum\lambda_z\delta_z$ of $\BB_n(X)$ is given as follows \cite{DJ}. If the support has at least two points, let $\delta=\min\{d(x,y)\mid x,y\in \supp\mu, \hspace{1mm} x\ne y\}$; for a one-point support, choose any $\delta>0$. Then for $\varepsilon<\delta/2$, the open set $U(\mu,\varepsilon)$, defined below as 
\[
\begin{aligned}
U(\mu,\varepsilon):=\Biggl\{\mu'=\sum_x\lambda'_x\delta_x\in\BB_n(X)\ \Biggm|\ 
&\lambda_z<\sum_{x\in B(z,\varepsilon)}\lambda'_x+\frac{\varepsilon}{n}\\[-1mm]
&\text{for every }z\in\supp\mu\Biggr\}
\subset B_\rho(\mu,\varepsilon).
\end{aligned}
\]
lies in the open ball $B_\rho(\mu,\varepsilon)$. The sets $U(\mu,\varepsilon)$ form a basis of the topology at $\mu$ in $\BB_n(X)$.

In \cite{KW}, the authors state that for metric spaces $X$, ``it seems likely that the Lévy-Prokhorov topology coincides with the quotient topology in the compact case but not in general'', and in their Lemma 8.4, they give evidence for this.  We check their statement below.

\begin{lem}
Let \(X\) be a compact metric space. Then the quotient topology on
$\BB_n(X)$
coincides with the topology induced from the Lévy--Prokhorov metric on the space
of probability measures of finite support.
\end{lem}

\begin{proof}
Set
$
A_n:=\coprod_{k=1}^n X^k\times \Delta^{k-1},
$
and let
$q:A_n\longrightarrow \BB_n(X)
$
be the quotient map.
We consider on the same underlying set \(\BB_n(X)\) also the subspace topology
induced from the Lévy--Prokhorov metric on the space of Borel probability
measures on \(X\). Denote this topological space by \(\BB_n(X)_{\rho}\), and denote
by \(\BB_n(X)_q\) the quotient-topologized space.

We claim that the map
\[
q:A_n\longrightarrow \BB_n(X)_{\rho},
\qquad
(x_1,\dots,x_k,t_1,\dots,t_k)\longmapsto \sum_{i=1}^k t_i\delta_{x_i},
\]
is continuous.
Indeed, since \(X\) is compact metric, the Lévy--Prokhorov topology agrees with
the weak\(^*\) topology on the space of Borel probability measures on \(X\); see \cite{Billingsley}.
Thus it suffices to check continuity against continuous test functions
\(f\in C(X)\). But for each \(k\),
\[
A_n \supset X^k\times \Delta^{k-1}\longrightarrow \mathbb R,
\qquad
(x_1,\dots,x_k,t_1,\dots,t_k)\longmapsto
\int_X f\, d\!\left(\sum_{i=1}^k t_i\delta_{x_i}\right)
=\sum_{i=1}^k t_i f(x_i),
\]
is continuous, since addition and multiplication are continuous and \(f\) is
continuous. Hence each restriction
$\displaystyle 
X^k\times \Delta^{k-1}\longrightarrow \BB_n(X)_{\rho}$
is continuous, and therefore \(q\) is continuous on the disjoint union \(A_n\), so it descends to a continuous map on the quotient
\begin{equation}\label{identity}
\mathrm{id}:\BB_n(X)_q\longrightarrow \BB_n(X)_{\rho}
\end{equation}
which is the identity on the underlying set.
Now \(A_n\) is compact, then so is \(\BB_n(X)_q\).
On the other hand, \(\BB_n(X)_{\rho}\) is metrizable, hence Hausdorff.
The map \eqref{identity}
is clearly bijective, since both spaces have the same underlying set. A continuous
bijection from a compact space to a Hausdorff space is a homeomorphism.
Therefore the quotient topology and the Lévy--Prokhorov topology on \(\BB_n(X)\)
coincide.
\end{proof}

\section{Boundary barycenter spaces}
\label{sec:def}

Let \(M\) be a compact connected smooth manifold of dimension at least two with nonempty boundary.  For \(p,q\in\N\cup\{0\}\), \(p+q\ge1\), define the open mixed stratum
\begin{equation}\label{eq:Bpq-open}
\BB_{p,q}^{\circ}(M,\pa M):=
\left\{
\sum_{i=1}^{p}\alpha_i\delta_{a_i}
+
\sum_{j=1}^{q}\beta_j\delta_{b_j}:
\begin{array}{l}
 a_i\in\mathring M,
 \quad b_j\in\pa M,\\
 a_i\ne a_{i'}\ (i\ne i'),
 \quad b_j\ne b_{j'}\ (j\ne j'),\\
 \alpha_i,\beta_j>0,
 \quad \sum_i\alpha_i+\sum_j\beta_j=1
\end{array}
\right\}.
\end{equation}
This is the stratum with \textit{$p$ interior support points and $q$ boundary support points}.
At the endpoints, for \(p\ge1\) and \(q\ge1\), respectively, the corresponding distinctness condition is omitted, and
$$\BB_{p,0}^\circ (M,\partial M) = \BB_p(\mathring{M})\setminus \BB_{p-1}(\mathring M)\ \ ,\ \ \BB_{0,q}^\circ (M,\partial M) = \BB_q(\partial M)\setminus \BB_{q-1}(\partial M)$$
The set \eqref{eq:Bpq-open} is topologized as a subspace of $\BB_{p+q}(M)$. It is convenient to refer to a point in $\BB_{p+q}(M)$ as a \textit{configuration}. The coefficients $\alpha_i$ and $\beta_j$ of a configuration in $\BB_{p,q}^\circ (M,\pa M)$ are called the \textit{weights}, and the collection of slots $\{a_1,\ldots, a_p,b_1,\ldots, b_q\}$ is called the \textit{effective support} of the configuration.

It is convenient at this stage to introduce the following closed strata: for \(p,q\ge0\) with \(p+q\ge1\), set
\begin{eqnarray}\label{eq:closed-strata}
\BB_q^p(M)&:=&
\left\{
\mu\in \BB_{p+q}(M):
\bigl|\supp(\mu)\cap\mathring M\bigr|\le p
\right\}\\
&=&\bigcup_{\substack{p',q'\ge0\\1\le p'+q'\le p+q,\ p'\leq p}} \BB_{p',q'}^{\circ}(M,\pa M)\nonumber
\end{eqnarray}
Evidently $
\BB_{q+i}^{p-i}(M)\subset \BB_q^p(M)$, for $0\le i\le p$, since each time an interior slot is allowed to lie on the boundary, one interior slot is converted into one boundary slot. At the endpoints,
$$\BB_q^0(M) = \BB_q(\pa M)\quad(q\ge1),
\qquad
\BB_0^p(M) = \BB_p(M)\quad(p\ge1).$$

Let us now recall that a space is \textit{locally closed} in a metric space if it is open in its closure. A \textit{decomposition} of a space $X$ is a partition of $X$ into a disjoint union of locally closed subspaces.

\begin{lem}
Assume $M$ is compact with boundary $\pa M$, $\dim M\geq 2$ and $\pa M\neq\emptyset$. We have the decomposition
    $$\displaystyle \BB_n(M) = \bigsqcup_{\substack{p,q\ge0\\1\le p+q\leq n}} 
\BB_{p,q}^{\circ}(M,\pa M)$$
In particular, $\BB_q^p(M)$ is the closure of 
$\BB_{p,q}^{\circ}(M,\pa M)$ in
$\BB_{p+q}(M)$ and, for \(p,q\ge1\),
$$\BB_q^p(M)\setminus \BB_{p,q}^{\circ}(M,\pa M) = \BB_{q+1}^{p-1}(M)\cup \BB^p_{q-1}(M).$$
At the endpoints the corresponding formulas are
\[
\BB_q^0(M)\setminus \BB_{0,q}^{\circ}(M,\pa M)=\BB_{q-1}^0(M),
\qquad
\BB_0^p(M)\setminus \BB_{p,0}^{\circ}(M,\pa M)=\BB_1^{p-1}(M).
\]
\end{lem}

\begin{proof}
For distinct pairs $(p,q)$, the spaces $
\BB_{p,q}^{\circ}(M,\pa M)$ are obviously disjoint and, as
\(1\le p+q\le n\), exhaust $\BB_n(M)$.  We now analyze the closure of
each $\BB_{p,q}^{\circ}(M,\pa M)$ in $\BB_{p+q}(M)$.
Let a sequence in \(\BB_{p,q}^{\circ}(M,\pa M)\) converge in \(\BB_{p+q}(M)\).  In the limit, some weights may become zero, support points may collide, and some interior points may converge to \(\pa M\).  None of these operations can create more than \(p\) effective interior support points.  Thus every limit belongs to \(\BB_q^p(M)\).

Conversely, let \(\sigma\in\BB_q^p(M)\).  Write \(\sigma\) with effective interior support \(a_1,\ldots,a_r\), \(r\le p\), and effective boundary support \(b_1,\ldots,b_s\), with \(r+s\le p+q\).  If fewer than \(p+q\) slots are present, split some positive weights into several nearby positive weights; this does not change the limit.  If more than \(q\) effective boundary slots are needed, use a collar neighborhood of \(\pa M\) and move the excess boundary slots a distance \(\epsilon>0\) into \(\mathring M\).  After also separating coincident points by distances \(O(\epsilon)\), we obtain an element of \(\BB_{p,q}^{\circ}(M,\pa M)\) converging to \(\sigma\) as \(\epsilon\to0\).  Hence \(\BB_q^p(M)=\overline{\BB_{p,q}^{\circ}(M,\pa M)}\).

Finally, the displayed boundary formulas are closed subsets of $\BB_q^p(M)$. Thus $\BB_{p,q}^{\circ}(M,\pa M)$ is open in its closure and is locally closed in $\BB_{p+q}(M)$.
\end{proof}

\begin{defn}\rm The boundary-weighted barycenter space of order \(l\ge1\) is
\begin{equation*}
\BB_l^{\pa}(M):=
\bigcup_{0<2p+q\le l}\BB_{p,q}^{\circ}(M,\pa M).
\end{equation*}
This is topologized as a subspace of $\BB_l(M)$.  Thus the union is disjoint by effective support type, while its closure relations are encoded by lower strata.
\end{defn}

\begin{lem}\label{lem:closed-strata} $\displaystyle
\BB_l^{\pa}(M)=\bigcup_{0<2p+q\le l}\BB_q^p(M).$
\end{lem}

\begin{proof} If \(\sigma\in\BB_q^p(M)\) and \(2p+q\le l\), then its effective support has, say, \(r\le p\) interior points and \(s\le p+q-r\) boundary points.  Hence \(2r+s\le p+q+r\le2p+q\le l\), so \(\sigma\) belongs to the open-stratum union defining \(\BB_l^{\pa}(M)\).
\end{proof}

\begin{exa}\rm \label{rem:not-homotopy-invariant} We consider $M=D$  the closed unit disk. According to Lemma \ref{lem:closed-strata}, we can write
$$
\BB_2^{\pa}(D)= \BB_2^0(D)\cup \BB^1_0(D) =\BB_2(\pa D)\cup \BB_1(D) $$
and these intersect as follows
$\BB_2(\pa D)\cap \BB_1(D)=\BB_1(\pa D)=\pa D = S^1.$ This pushout description is represented diagrammatically by the following diagram
$$\xymatrix{
\BB_2^0=\BB_2(\pa D)&
\BB_1^0=\pa D\ar[l]\ar[r]&
\BB_0^1=\BB_1(D)}$$
From this pushout, the homotopy type is straightforward. Since \(\BB_2(\pa D) = \BB_2(S^1)\simeq S^3\) \cite[Corollary 1.4(b)]{KallelKaroui2011}, the space $\BB_2^\partial (D)$ is therefore obtained by gluing the disk $D=\BB_1(D)$ to a circle embedded in \(S^3\).  In other words, $\BB_2^\partial (D)$ is the mapping cone of a tame embedding $S^1\hookrightarrow S^3$, and up to homotopy 
$\BB_2^\partial (D)\simeq S^3/S^1\simeq S^3\vee S^2$.
The homology of this space does not depend on the embedding and is given by
\[
\wt H_r(\BB_2^{\pa}(D);\Z)\cong H_r(S^3,S^1;\Z)\cong
\begin{cases}
\Z,& r=2,3,\\
0,& \text{otherwise}.
\end{cases}
\]
More on the homology of $\BB_l^\partial (D)$ in later sections.
\end{exa}

\begin{exa}\rm 
We can describe $\BB_3^\pa (M)$ similarly. A configuration in this space must have at most one point in $\mathring{M}$ (since its cost is $2$) or all three points in the boundary. This means that
$$\BB_3^\pa (M) = \BB^1_1(M)\cup \BB_3(\partial M)$$
Both factors intersect at 
$\BB^1_1(M)\cap \BB_3(\partial M) = \BB_2(\pa M)$. This is a pushout that can be conveniently described by a similar triangular diagram as in the previous example.
\end{exa}

\begin{rem}\rm 
Let \(m=\dim M\).  Then the stratum $\BB_{p,q}^{\circ}(M,\pa M)$ with \(p\) interior points and \(q\) boundary points is an open manifold of dimension
\[
pm+q(m-1)+(p+q-1)=p(m+1)+qm-1
\].
Under the constraint \(2p+q\le l\), this is maximized by \(p=0\), \(q=l\) when $m>1$.  Since $\BB_l^\pa (M)$ is a CW complex, it follows that as a CW complex, $\dim \BB_l^{\pa}(M)\le lm-1$ so that
\begin{equation*}
H_r(\BB_l^{\pa}(M);\F)=0\quad\text{for }r\ge lm
\end{equation*}
This vanishing is verified by our homology decomposition in \S\ref{sec:general-homology}.
\end{rem}

\section{The triangular meet-semilattice}
\label{subsec:filtration-graphics}

We reformulate Lemma \ref{lem:closed-strata} as a colimit of an explicit pushout diagram.
It comes in the form of an intersection lattice where the meet of any two subspace entries is their intersection. Remarkably, this lattice is a \textit{triangular meet-semilattice}, that is it satisfies the extra condition spelled out in Lemma \ref{meet}. The lattice has the triangular shape depicted in Fig. \ref{fig:triangular-filtration}.

\begin{lem}\label{lem:compatible-triangulations}
For a fixed order \(l\), the finite family consisting of the spaces
\(\BB_n(M)\), the closed strata \(\BB_q^p(M)\), and their intersections
admits compatible CW structures in which every inclusion occurring in
\(\mathcal D_l\) is the inclusion of a subcomplex. Consequently all these
inclusions are closed cofibrations, and the diagram \(\mathcal D_l\) is Reedy
cofibrant\footnote{Let $P$ be a poset. A diagram $
X\colon P\longrightarrow \mathbf{Top}
$
is Reedy cofibrant if every ``latching map'' $\colim_{\beta<\alpha}X_\beta\rightarrow X_\alpha$ is a  cofibration.}.
\end{lem}

\begin{proof}
(sketch, based on Appendix of \cite{malchiodi}). Choose a finite triangulation of the pair \((M,\pa M)\). For \(n\le l\), use
the standard finite symmetric-join cell structure associated with the quotient model
\[
\coprod_{k=1}^{n}M^k\times\Delta^{k-1}\longrightarrow\BB_n(M).
\]
After a common barycentric subdivision, the permutation actions, the face
maps obtained by deleting a zero weight, and the diagonal maps obtained by
merging coincident points are cellular. In this cell structure, the conditions
\emph{at most \(p\) interior support points} and \emph{at most \(p+q\) support
points} are unions of closed faces. Passing to the finite quotient gives
compatible CW structures on all \(\BB_q^p(M)\) and on their finite
intersections. The inclusion of a subcomplex is a closed
cofibration \cite[Chapter~5]{Strom2011}.
Finally, every latching space is a union of lower closed strata and hence a
subcomplex of the corresponding object, which is precisely Reedy
cofibrancy.
\end{proof}

Consider the  diagrams of spaces $\mathcal D_2,\mathcal D_3$ and $\mathcal D_4$ as in Figure \ref{fig:low-order-diagrams}. These diagrams keep track of the closure relations, and the arrows pointing up are cofibrations. All subspaces figuring in the diagrams are compact triangulable spaces 
and the colimit is identified in this case with the union of the displayed pieces glued along their displayed intersections. The colimits of these diagrams are $\BB_l^\pa (M)$, for $l=2,3,4$, and this is a restatement of the equality $\displaystyle
\BB_l^{\pa}(M)=\bigcup_{0<2p+q\le l}\BB_q^p(M)$ (Lemma \ref{lem:closed-strata}).

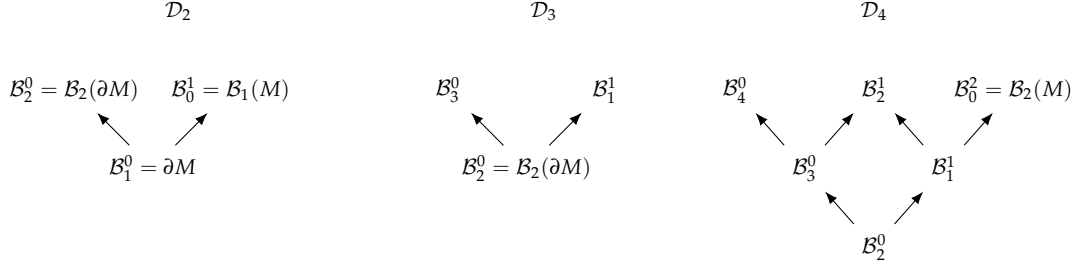
\begin{figure}[htbp]
\centering
\begin{tikzpicture}[>=Latex,every node/.style={font=\scriptsize},x=1.15cm,y=1.3cm]
\node at (-4.0,0.8) {$\mathcal D_2$};
\node (a) at (-5.2,0) {$\BB_2^0=\BB_2(\pa M)$};
\node (b) at (-3.4,0) {$\BB_0^1=\BB_1(M)$};
\node (c) at (-4.3,-0.8) {$\BB_1^0=\pa M$};
\draw[->] (c) -- (a); \draw[->] (c) -- (b);
\node at (0.2,0.8) {$\mathcal D_3$};
\node (d) at (-0.9,0) {$\BB_3^0$};
\node (e) at (0.9,0) {$\BB_1^1$};
\node (f) at (0,-0.8) {$\BB_2^0=\BB_2(\pa M)$};
\draw[->] (f) -- (d); \draw[->] (f) -- (e);
\node at (4.0,0.8) {$\mathcal D_4$};
\node (g) at (2.4,0) {$\BB_4^0$};
\node (h) at (4.0,0) {$\BB_2^1$};
\node (i) at (5.6,0) {$\BB_0^2=\BB_2(M)$};
\node (j) at (3.2,-0.8) {$\BB_3^0$};
\node (k) at (4.8,-0.8) {$\BB_1^1$};
\node (l) at (4.0,-1.6) {$\BB_2^0$};
\draw[->] (j) -- (g); \draw[->] (j) -- (h);
\draw[->] (k) -- (h); \draw[->] (k) -- (i);
\draw[->] (l) -- (j); \draw[->] (l) -- (k);
\end{tikzpicture}
\caption{The first three diagrams. By design, the colimit of
$\mathcal D_l$ is $\BB_l^\partial(M)$.}
\label{fig:low-order-diagrams}
\end{figure}

To construct the general diagram $\mathcal D_l$ for $l\geq 2$, we make the following useful observation 
\begin{equation}\label{formulaintersection}
\BB_q^p(M)\cap \BB_r^s(M)= \BB^{\min \{p,s\}}_{\min \{p+q,r+s\}-\min \{p,s\}} (M)
\end{equation}
Indeed, a configuration in $\BB_q^p(M)$ has at most \(p+q\) slots, at most \(p\) of which are in the interior, while a configuration in $\BB_r^s(M)$ has at most \(r+s\) slots, at most \(s\) of which are in the interior. A configuration in their intersection therefore has at most $\min\{p,s\}$ interior slots and at most $\min\{p+q,r+s\}$ slots in total.

To construct the meet-semilattice in general, we start with the maximal subspaces.
For \(l\ge1\), the space \(\BB_{2l}^{\pa}(M)\) is the union of its maximal subspaces
\[
\BB_{2l}^{0},\quad
\BB_{2l-2}^{1},\quad
\ldots,\quad
\BB_2^{l-1},\quad
\BB_0^l=\BB_l(M),
\]
This we take to be the top row of the diagram. The next row consists of the pairwise intersections
\[
\BB_{2l-1}^{0},\quad
\BB_{2l-3}^{1},\quad
\ldots,\quad
\BB_1^{l-1}
\]
with all lower rows obtained by continuing the same pattern. 
Clearly, the entries in a fixed row are of the form
$\BB_n^0,
\BB_{n-2}^1,
\BB_{n-4}^2,
\ldots,\quad
\BB_{n-2j}^j,\quad
\ldots$, and
the basic adjacent-intersection condition is
$
\BB_{n-2j}^{\,j}
\cap
\BB_{n-2j-2}^{\,j+1}
=
\BB_{n-2j-1}^{\,j}
$.

To obtain the full triangular meet-semilattice, it is enough to require the following interval-intersection property.

\begin{lem}\label{meet}
Repeated adjacent intersections satisfy
\[
\bigcap_{j=p}^{q}\BB_{n-2j}^{\,j}
=
\BB_{\,n-p-q}^{\,p},
\]
so that, equivalently,
$\BB_{n-2p}^{\,p}
\cap
\BB_{n-2q}^{\,q}
=
\BB_{\,n-p-q}^{\,p}$, $
p\le q$.
\end{lem}

\begin{proof}
We must check that
\[
\BB_{n-2p}^{\,p}
\cap
\BB_{n-2q}^{\,q}
=
\bigcap_{j=p}^{q}\BB_{n-2j}^{\,j},
\qquad
p<q.
\]
But both sides coincide with $\BB_{n-p-q}^p(M)$ according to \eqref{formulaintersection}.
\end{proof}

\begin{prop}\label{colimitdiagram}
The colimit diagram $\mathcal D_{2l}$ for $\BB_{2l}^{\pa}(M)$ is depicted below.  The diagram for \(\BB_{2l-1}^{\pa}(M)\) is obtained by truncating the top row. In all cases $\colim\mathcal D_l=\BB_l^\partial (M)$.
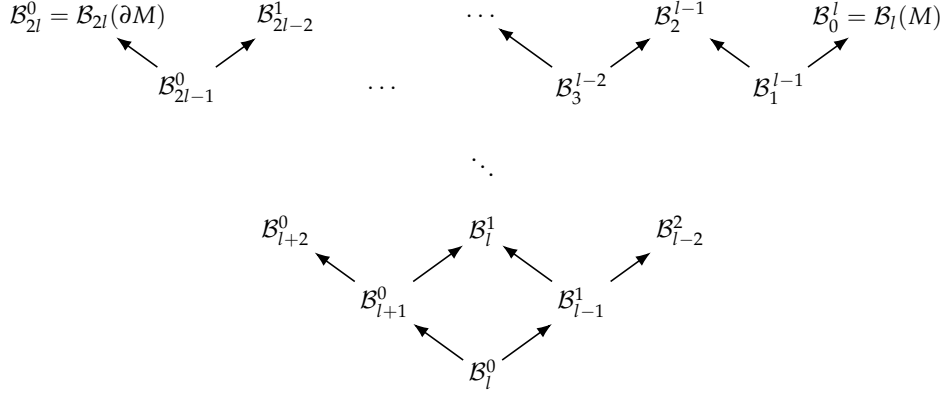
\begin{figure}[H]
\centering
\begin{tikzpicture}[
scale=0.9,
    transform shape,
    >=Latex,
    every node/.style={
        font=\small,
        inner sep=2pt
    },
    x=1.45cm,
    y=1.05cm,
    every path/.style={line width=0.6pt}
]

\node (a0) at (-4,0)
    {$\BB_{2l}^{0}=\BB_{2l}(\partial M)$};

\node (a1) at (-2,0)
    {$\BB_{2l-2}^{1}$};

\node (a2) at (0,0)
    {$\cdots$};

\node (a3) at (2,0)
    {$\BB_{2}^{\,l-1}$};

\node (a4) at (4,0)
    {$\BB_{0}^{\,l}=\BB_l(M)$};

\node (b0) at (-3,-1)
    {$\BB_{2l-1}^{0}$};

\node (b1) at (-1,-1)
    {$\cdots$};

\node (b2) at (1,-1)
    {$\BB_{3}^{\,l-2}$};

\node (b3) at (3,-1)
    {$\BB_{1}^{\,l-1}$};

\node (c0) at (0,-2)
    {$\ddots$};

\node (d0) at (-2,-3)
    {$\BB_{l+2}^{0}$};

\node (d1) at (0,-3)
    {$\BB_{l}^{1}$};

\node (d2) at (2,-3)
    {$\BB_{l-2}^{2}$};

\node (e0) at (-1,-4)
    {$\BB_{l+1}^{0}$};

\node (e1) at (1,-4)
    {$\BB_{l-1}^{1}$};


\node (f0) at (0,-5)
    {$\BB_{l}^{0}$};

\draw[->] (b0) -- (a0);
\draw[->] (b0) -- (a1);

\draw[->] (b2) -- (a2);
\draw[->] (b2) -- (a3);

\draw[->] (b3) -- (a3);
\draw[->] (b3) -- (a4);

\draw[->] (e0) -- (d0);
\draw[->] (e0) -- (d1);

\draw[->] (e1) -- (d1);
\draw[->] (e1) -- (d2);

\draw[->] (f0) -- (e0);
\draw[->] (f0) -- (e1);

\end{tikzpicture}
\caption{The triangular meet-semilattice $\mathcal D_{2l}$ whose colimit is $\BB_{2l}^{\pa}(M)$.}
\label{fig:triangular-filtration}
\end{figure}
\end{prop}

\begin{proof}
For \(\BB_{2l}^{\pa}(M)\) this is clear. For \(\BB_{2l-1}^{\pa}(M)\), the maximal weighted cost is \(2l-1\); hence the piece \(\BB_0^l= \BB_l(M)\), which has cost \(2l\), is absent, and the corresponding lower intersections are also removed.  This corresponds to the truncated diagram.
\end{proof}

Proposition \ref{colimitdiagram} has two immediate and useful consequences: connectivity and the Euler characteristic computation. By Lemma \ref{lem:compatible-triangulations}, the diagram above is Reedy cofibrant, so the colimit is also the homotopy colimit (see the proof of Lemma \ref{colimitconnectivity}).

\subsection{Connectivity}
\label{sec:connectivity}

Recall that a space is \(r\)-connected if its homotopy groups vanish up to and including degree $r$, where $r\geq 0$. The space is simply connected if it is $1$-connected. The largest such $r$ is called the connectivity of $X$ and written $\operatorname{conn}(X)$. In the case $X$ is contractible, $\conn (X)=+\infty$.

We now review some basic facts on pushouts and pushout diagrams. The colimit of the simplest triangle has the following description
\begin{eqnarray*}B\sqcup_AC &=&  (B\sqcup C)\big/\bigl(f(a)\sim g(a),\ a\in A\bigr) \\
&=:&\colim(\xymatrix{C&A\ar[r]^g\ar[l]_f&B)}
\end{eqnarray*}
This is an identification space and topologists use the language of a ``pushout square'' or just ``pushout'' to describe this colimit
\begin{equation}\label{pushoutsquare}\xymatrix{A\ar[r]^f\ar[d]_g&B\ar[d]\\
C\ar[r]&D} \ \hbox{is a pushout square}\ \Longleftrightarrow D=\colim (\xymatrix{C&A\ar[r]^f\ar[l]_{g}&B)}
\end{equation}

When $(B,A)$ is an NDR pair\footnote{A Neighborhood Deformation Retract. This is equivalent to the inclusion $A\hookrightarrow B$ being a cofibration.}, this colimit has homotopy invariance, and in fact has the homotopy type of the double mapping cylinder associated to the maps $f$ and $g$. This double mapping cylinder construction is known as the ``homotopy colimit'' (or ``homotopy pushout'') and is written as $D=\hbox{hcolim}(C\leftarrow A\rightarrow B)$. So when either $(B,A)$ or $(C,A)$ is an NDR pair in the pushout square, both colim and hcolim (or double mapping cylinder) coincide, up to homotopy. General references on this material and on homotopy invariance are \cite{piacenza, Strom2011}. 

\begin{lem}\label{equivalencequotients} For a pushout square as in \eqref{pushoutsquare}, $A\hookrightarrow B$ and $C\hookrightarrow D$ cofibrations, we have an equivalence of quotient spaces
$\displaystyle D/C \simeq B/A$.
\end{lem}

The following result on building pushouts is also widely used in practice.

\begin{lem}\label{pushoutmethod}  Let \(i: A\hookrightarrow B\), let $f\colon B\longrightarrow D$ be a surjective quotient map, and $C\hookrightarrow D$ a subspace. Assume that $f^{-1}(C)=A$ and that the restriction $ f|_{B\setminus A}\colon B\setminus A\longrightarrow D\setminus C $ is injective. Then \[ D\cong \operatorname*{colim} \left( C\xleftarrow{\,f|_A\,}A\xrightarrow{\,i\,}B \right). \]
If $(B,A)$ is an NDR-pair, then $D$ is homotopy equivalent to the homotopy pushout.
\end{lem}

Being a homotopy pushout allows us to derive most of the general connectivity properties listed below: 
\begin{itemize}
    \item If $X$ is $r$-connected, $r\geq 1$, then \(\BB_n(X)\) is \((2n+r-2)\)-connected \cite[Theorem 1.2]{KallelKaroui2011} 
    \item \(\conn(X*Y)\ge \conn(X)+\conn(Y)+2\)
    \item  Assume $(B,A)$ is an NDR pair, \(A\neq\varnothing\), and that the spaces have the homotopy type of CW-complexes. From  the long exact sequence of the pair \((B,A)\) one has that
\[\conn(B/A)\ge \min\{\conn (B),\conn (A)+1\}.
\]
\item Let $D$ be the homotopy pushout $D = \hbox{hcolim}(C\leftarrow A\rightarrow B)$. Using the two previous bullet points, we have
\begin{eqnarray}\label{connectivitypushout}
\conn(D)&\ge& 
\min\{\conn(C),\,\conn(D/C)\}\nonumber\\
&=& \min\{\conn(C),\,\conn(B/A)\}\\
&\geq&
\min\{\conn(C),\,\conn(B),\,\conn(A)+1\}\nonumber
\end{eqnarray}
\end{itemize}

\begin{prop}\label{connectivityBpq}
For \(p\ge1, q\ge 1\), there is a strict pushout square 
\begin{equation}\label{eq:Bpq-colimit}
\xymatrix{
\BB_1^{p-1}(M)*\BB_q(\pa M)
\ar[r]\ar[d]&
\BB_p(M)*\BB_q(\pa M)\ar[d]\\
\BB_{q+1}^{p-1}(M)\ar[r]&
\BB_q^p(M)}
\end{equation}
This is also a homotopy pushout. In particular, for all
\((p,q)\in(\N\cup\{0\})^2\) with \(p+q\ge1\),
$\BB_q^p(M)$ is at least $2(p+q)+r-2$-connected if both
$M$ and $\pa M$ are $r$-connected, $r\geq 1$.
\end{prop}

\begin{proof}
If $p=0$, $\BB_q^0(M) = \BB_q(\pa M)$ has the right connectivity, and if $q=0$, $\BB_0^p(M)=\BB_p(M)$ also has the predicted connectivity. We then assume $p\geq 1, q\geq 1$, and consider the map 
\begin{eqnarray*}
f: \BB_p(M)*\BB_q(\pa M)&\longrightarrow& \BB_q^p(M)\\
\ [\mu,\nu, s]&\longmapsto& s\mu + (1-s)\nu
\end{eqnarray*}
This is a well-defined surjective quotient map, and it is injective on the complement of $\BB^{p-1}_1(M)*\BB_q(\partial M)$
in $ \BB_p(M)*\BB_q(\partial M)$. Indeed, outside of $\BB^{p-1}_1(M)*\BB_q(\partial M)$, the interior barycenter has exactly $p$ distinct interior support points; hence the interior and boundary parts of the resulting barycenter are uniquely determined, and f is injective there.
Moreover, $f^{-1}(\BB^{p-1}_{q+1}(M))$ is precisely
$\BB^{p-1}_1(M)*\BB_q(\partial M)$. We can now apply
Lemma \ref{pushoutmethod} to conclude that  \eqref{eq:Bpq-colimit} is a pushout square. Since this is also a (homotopy) pushout, we can invoke \eqref{connectivitypushout} to deduce the inequality
\[
\begin{aligned}
\conn(\BB^p_q(M))\geq \min\bigl\{&\conn(\BB^{p-1}_{q+1}(M)),
\conn(\BB_p(M)*\BB_q(\pa M)),\\
&\conn(\BB_1^{p-1}(M)*\BB_q(\pa M))+1\bigr\}.
\end{aligned}
\]
We will evaluate each connectivity. We have that
$$\conn (\BB_p(M)*\BB_q(\partial M))\geq (2p+r-2) + (2q+r-2) + 2= 2(p+q)+ 2r-2$$
We proceed by induction. The first step is to estimate connectivity of $ \BB_1^{p-1}(M)\subset \BB_p(M)$. For a general closed pair $(X,A)$, we will write the image of $x$ in $X/A$ by $[x]$. Let $W_0$ be the subspace of configurations in $\BB_p(M/\pa M)$ containing the basepoint $x_0\in M/\pa M$. We have a well-defined map
\begin{eqnarray*}
\BB_p(M)/\BB_1^{p-1}(M)&\longrightarrow&
\BB_p(M/\partial M)/W_0\\
\left[\sum \alpha_ix_i\right]&\longmapsto&
\left[\sum \alpha_i[x_i]\right]
\end{eqnarray*}
and this is a bijection between compact Hausdorff spaces, so it is a homeomorphism.
But for a path-connected space, $W_0$ is contractible by the contraction 
$$I\times W_0\longrightarrow W_0\ ,\ 
\left(t, sx_0 + \sum t_iy_i\right)\longmapsto \left(s + t\sum t_i\right)x_0 + (1-t)\sum t_iy_i$$
This is well-defined since
the coefficients sum up to one, and it is a contraction.
Since $W_0$ is contractible, and $(\BB_p(M/\pa M),W_0)$ is an NDR pair, we have the homotopy equivalence
\begin{equation}\label{quotient} \BB_p(M)/\BB_1^{p-1}(M)\simeq 
\BB_p(M/\partial M)
\end{equation}
For \(p=1\), both sides are exactly the quotient \(M/\pa M\).

We assumed that $M$ and $\partial M$ are $r$-connected, which implies by the long exact sequence in homotopy groups that $M/\pa M$ is also $r$-connected, and thus $\BB_p(M/\pa M)$ is $2p+r-2$-connected. Since $\BB_p(M)$ is also $2p+r-2$-connected, it follows that
$$\BB_1^{p-1}(M)\ \hbox{is (at least) $2p+r-3$-connected}$$
and thus
$\BB_1^{p-1}(M)*\BB_q(\pa M)$
is at least $(2p+r-3) + (2q+r-2)+2 = 2(p+q)+2r-3$-connected.
We now have all the ingredients to apply inequality \eqref{connectivitypushout}. Set
\[
A=\BB_1^{p-1}(M)*\BB_q(\partial M),\qquad
B=\BB_p(M)*\BB_q(\partial M),
\]
\[
C=\BB_{q+1}^{p-1}(M),\qquad
D=\BB_q^p(M).
\]
and assume the induction statement
$$P(p):\conn(\BB_q^p(M))\ge 2(p+q)+r-2\ \hbox{for every}\ q\geq 0$$ This statement is true for $p=0$ since $\BB_q^0(M)=\BB_q(\partial M)$ and $\partial M$ is $r$-connected.
We can assume $P(p-1)$ is true, and use \eqref{eq:Bpq-colimit} to conclude $P(p)$. More precisely,
substitute the connectivity estimates already established to the spaces above
\[
\conn(A)
=
\conn\!\left(\BB_1^{p-1}(M)*\BB_q(\partial M)\right)
\ge
2(p+q)+2r-3.
\]
\[
\conn(B)
=
\conn\!\left(\BB_p(M)*\BB_q(\partial M)\right)
\ge
2(p+q)+2r-2.
\]
and
$
\conn(C)
=
\conn\!\left(\BB_{q+1}^{p-1}(M)\right)
\ge
2(p+q)+r-2
$ by the induction hypothesis.
Therefore,
\[
\begin{aligned}
\conn(D)
&\ge
\min\Bigl\{
2(p+q)+r-2,\,
2(p+q)+2r-2,\,
2(p+q)+2r-2
\Bigr\}\\
&=
2(p+q)+r-2,
\end{aligned}
\]
provided \(r\ge 1\). This is our claim.
\end{proof}

We will now derive our connectivity result from the meet-semilattice description of the colimit diagram for $\BB_l^\pa(M)$. This can be done by analyzing filtration terms and their quotients. We opt to use a more general but very useful result which we state next for convenience.

\begin{lem}\label{colimitconnectivity}
Let \(P\) be a finite non-empty meet-semilattice, regarded as a category, and let
$
X\colon P\longrightarrow \mathbf{Top}
$
be a diagram that is Reedy cofibrant. Assume that all the
spaces \(X_\alpha\) have the homotopy type of CW-complexes. Then
\[
\conn\left(\operatorname*{colim}_{\alpha\in P}X_\alpha\right)
\geq
\min_{\alpha\in P}\conn(X_\alpha).
\]
\end{lem}

\begin{proof}
Since the diagram is Reedy cofibrant, the canonical map
$
\operatorname*{hocolim}_{\alpha\in P}X_\alpha
\longrightarrow
\operatorname*{colim}_{\alpha\in P}X_\alpha
$
is a weak homotopy equivalence, and since all spaces are of the homotopy type of CW complex, it is an actual homotopy equivalence \cite{piacenza}.
Put
$
r=\min_{\alpha\in P}\conn(X_\alpha)
$.
The natural transformation \(X\to *\) induces a map
\[
\Phi : \operatorname*{hocolim}_{\alpha\in P}X_\alpha
\longrightarrow
\operatorname*{hocolim}_{\alpha\in P}*
\simeq |N(P)|.
\]
Each map $X_\alpha\rightarrow *$ is $r+1$-connected, and thus
by the connectivity theorem for geometric realizations of simplicial
spaces (see \cite{ebertwilliams}, Lemma 2.4), the map $\Phi$ is \((r+1)\)-connected.
Since \(P\) is a finite meet-semilattice, it has a least element, namely
the meet of all its elements. If we view $P$ as a small category, then this least element is an initial element. This implies that the classifying space \(|N(P)|\) is contractible. It
follows that
$
\operatorname*{hocolim}_{\alpha\in P}X_\alpha
$
is at least \(r\)-connected, and hence so is $
\operatorname*{colim}_{\alpha\in P}X_\alpha$.
\end{proof}

We will apply this lemma to our finite meet-semilattice. Another proof would filter the space by level height (see \eqref{filtration}), obtain a sequential colimit, and compute the connectivity of consecutive quotients.

\begin{cor}\label{prop:connectivity} (Theorem \ref{thm:intro-euler})
Assume that both \(M\) and \(\pa M\) are \(r\)-connected, $r\geq 1$. Then 
\begin{equation}\label{eq:connectivity-corollary}
\BB_l^{\pa}(M)\text{ is at least}\ \begin{cases}(l+r-2)\text{-connected}& \hbox{if $l$ even}\\
(l+r-1)\text{-connected}& \hbox{if $l$ odd}
\end{cases}
\end{equation} 
The space $\BB_l^\pa (M)$ is simply connected as soon as $l\ge 2$.
\end{cor}

\begin{proof}
Consider the triangular meet-semilattice diagram $\mathcal D_{2l}$ in Fig. \ref{fig:triangular-filtration} describing $\BB_{2l}^\pa (M)$.
It is Reedy cofibrant because
$L_\alpha\mathcal D_{2l} = \bigcup_{\beta < \alpha} X_\beta\subset X_\alpha$ is a subcomplex, for every $X_\alpha\in\mathcal D_{2l}$.
Now both \(M\) and \(\pa M\) are \(r\)-connected,
so each of the constituent subspaces $\BB_q^p(M)$ is $2(p+q)+r-2$-connected by Proposition \ref{connectivityBpq}.
By Lemma \ref{colimitconnectivity}, the connectivity of the colimit is at least the minimum connectivity of its constituent spaces
and so is at least $2(p+q)+r-2$ connected as well. 
The smallest such bound occurs when $p+q=l$ along the right-hand boundary of the diagram, and hence
\(\conn (\BB_{2l}^\pa (M))\geq 2l+r-2\). For the truncated odd diagram the same argument gives
\(\conn (\BB_{2l-1}^\pa (M))\ge 2l+r-2\). Replacing \(2l\) and \(2l-1\) by a single order variable gives exactly \eqref{eq:connectivity-corollary}.

Finally, if $l\ge 3$, obviously $l+r-2>0$ and $\BB_l^\pa (M)$ is simply connected. If $l=2$ and $r\geq 1$, the bound $l+r-2\geq 1$ and simple connectivity holds as well.
\end{proof}

\begin{rem}\rm 
    When $l=2$ and $\partial M$ is not simply connected, $\BB_l^\pa (M)=\BB_2^\pa (M)$ is not necessarily simply connected anymore. Let $M=T^\circ $ be the torus minus a small closed disk. Then $\pa M\cong S^1$. We have the pushout diagram
    \begin{eqnarray*}\BB_2^\pa (T^\circ) &=&\colim (\xymatrix{
\BB_2^0=\BB_2(\pa T^\circ )&
\BB_1^0= S^1 \ar[l]\ar[r]&
\BB_0^1=T^\circ})\\
&=&\colim (\xymatrix{
S^3& S^1 \ar[l]\ar[r]&T^\circ
})
\end{eqnarray*}
    and by an application of Seifert--van Kampen, $\pi_1(\BB_2^\pa (T^\circ))\cong\Z\times\Z$.
\end{rem}

\subsection{Euler characteristic in even dimension}
\label{sec:euler}

We apply inclusion--exclusion to the cover by the maximal entries in the top
row. By Lemma \ref{meet}, the intersection indexed by a nonempty set \(S\) of
top-row entries depends only on the interval from \(\min S\) to \(\max S\).
Fix an interval with \(d\) gaps. If \(d=0\), its coefficient is \(1\); if
\(d=1\), its coefficient is \(-1\). For \(d\ge2\), summing over all choices of
the \(d-1\) intermediate entries gives
\[
-\sum_{j=0}^{d-1}(-1)^j\binom{d-1}{j}=-(1-1)^{d-1}=0.
\]
Thus only the top row and the adjacent-pair row survive. If \(T_r\) denotes
the sum of the Euler characteristics of all entries in row \(r\), starting
with the maximal row \(r=l+1\) (respectively \(r=l\)) for \(\mathcal D_{2l}\)
(respectively \(\mathcal D_{2l-1}\)), then
\[
\chi (\colim \mathcal D_{2l}) = T_{l+1}-T_l,
\qquad
\chi (\colim \mathcal D_{2l-1}) = T_l-T_{l-1}.
\]

\begin{cor}\label{lem:euler-sum}
For \(l\ge1\),
\begin{align*}
\ch(\BB_{2l}^{\pa}(M))
&=
\sum_{i=0}^{l}\ch(\BB_{2l-2i}^{i})-
\sum_{i=0}^{l-1}\ch(\BB_{2l-1-2i}^{i}),\\
\ch(\BB_{2l-1}^{\pa}(M))
&=
\sum_{i=0}^{l-1}\ch(\BB_{2l-1-2i}^{i})-
\sum_{i=0}^{l-2}\ch(\BB_{2l-2-2i}^{i}).
\end{align*}
\end{cor}

\begin{proof}
For the even diagram $\mathcal D_{2l}$, the top row is \(\BB_{2l-2i}^{i}\), \(0\le i\le l\), and the next row is \(\BB_{2l-1-2i}^{i}\), \(0\le i\le l-1\) as already indicated, giving the first formula. The same argument applies to $\mathcal D_{2l-1}$.
\end{proof}

On our way to derive Theorem \ref{thm:intro-euler}, the  following two calculations are needed:
\begin{enumerate}
\item For joins,
\begin{equation}\label{eq:euler-join}
\ch(X*Y)=\ch(X)+\ch(Y)-\ch(X)\ch(Y).
\end{equation}
\item For ordinary barycenter spaces,
\begin{equation}\label{eq:ordinary-euler}
\ch(\BB_k(X))=
1-\frac1{k!}(1-\ch(X))(2-\ch(X))\cdots(k-\ch(X)).
\end{equation}
This is by \cite[Corollary 1.4(a)]{KallelKaroui2011}.
\end{enumerate}

\begin{lem}\label{lem:Bpq-euler}
Let \(M\) be compact and even-dimensional.  For every \(p,q\ge0\) with \(p+q\ge1\),
\begin{equation*}
\ch(\BB_q^p(M))=\ch(\BB_p(M)).
\end{equation*}
By convention $\chi (\BB_0(M))=0$.
\end{lem}

\begin{proof}
The case \(q=0\) of the lemma is equally immediate since
\(\BB_0^p(M)= \BB_p(M)\). We assume $q\geq 1$.
When \(M\) is compact and even-dimensional, then \(\pa M\) is closed and odd-dimensional, hence \(\ch(\pa M)=0\). According to \eqref{eq:ordinary-euler}, we see that
\begin{equation}\label{eq:boundary-bary-euler-zero}
\chi (\BB_q^0(M)) = \ch(\BB_q(\pa M))=0
\qquad q\ge1.
\end{equation}
This is in agreement with $\chi (\BB_0(M))=0$. 
So the claim is true for every $q\geq 0$ when $p=0$. We can now prove the claim
$$\mathcal P(p) : \chi(\BB_q^p(M))=\chi(\BB_p(M)) \ \ \hbox{for every $q\geq 0$}$$
by induction on $p$.
Assume $\mathcal P(p-1)$.
 By the additivity of the Euler characteristic on pushouts, using \eqref{eq:euler-join} and \eqref{eq:boundary-bary-euler-zero}, we get
\[
\ch(\BB_q^p)=
\ch(\BB_{q+1}^{p-1})+
\ch(\BB_p)-
\ch(\BB_1^{p-1}).
\]
In the formula above, the two terms involving \(p-1\) equal to
\(\chi(\BB_{p-1}(M))\) by the induction hypothesis, thus they cancel out. Hence
$\chi(\BB_q^p(M)) = \chi (\BB_p(M))$.
\end{proof}

\begin{proof}[Proof of Theorem \ref{thm:intro-euler}]
Combining Lemma \ref{lem:euler-sum} with Lemma \ref{lem:Bpq-euler}, we obtain
\[
\ch(\BB_{2l}^{\pa}(M))=
\sum_{i=0}^{l}\ch(\BB_i(M))-
\sum_{i=0}^{l-1}\ch(\BB_i(M))=\ch(\BB_l(M)),
\]
and similarly
\[
\ch(\BB_{2l-1}^{\pa}(M))=
\sum_{i=0}^{l-1}\ch(\BB_i(M))-
\sum_{i=0}^{l-2}\ch(\BB_i(M))=\ch(\BB_{l-1}(M)).
\]
If \(\ch(M)=0\), then \eqref{eq:ordinary-euler} gives \(\ch(\BB_r(M))=0\) for every \(r\ge1\), and \(\ch(\BB_0(M))=0\) by convention.
\end{proof}


\section{The homology decomposition}
\label{sec:general-homology}

In the introduction, we discussed the filtration of $\BB_l^\pa (M)$ by the $\BB_i^\pa (M)$ with $i\leq l$. 
For \(l\ge2\), the natural inclusion of $\BB_{l-1}^\pa (M)$ in $\BB_l^\pa (M)$ is a topological embedding and in fact a cofibration. This property is recorded below.

\begin{lem}\label{ndrpair} For every \(l\ge2\), $(\BB_l^\pa (M), \BB_{l-1}^\pa (M))$ is an NDR pair.
\end{lem}

\begin{proof}
For every pair \((p,q)\) satisfying
$
2p+q=l,
$ 
we use the augmented convention \(\BB_0(X)=S^{-1}\) only inside join
expressions, so that \(X*\BB_0(Y)=X\), and put
$$J_{p,q}=\begin{cases} \BB_p(M)*\BB_q(\partial M),& pq\neq 0\\
\BB_l(\pa M),& p=0\\
\BB_p(M),& l=2p\ \hbox{and}\ q=0
\end{cases}
$$
For each $(p,q)$, $pq\neq 0$, we consider the natural map
\[
\Phi_{p,q}\colon J_{p,q}\longrightarrow
\BB_l^\partial(M)\ ,\ 
\Phi_{p,q}([\sigma,\tau,s])
=
s\sigma+(1-s)\tau.
\]
It is easy to see that
\[
L_{p,q}
=
\Phi_{p,q}^{-1}
\left(
\BB_{l-1}^\partial(M)
\right) =
\left(
\BB_1^{p-1}(M)*\BB_q(\partial M)
\right)
\cup
\left(
\BB_p(M)*\BB_{q-1}(\partial M)
\right)
\subseteq J_{p,q}.
\]
At the two endpoints we set
\(\Phi_{0,l}\) the natural inclusion, $
L_{0,l}=\BB_{l-1}(\partial M)$, and when \(l=2p\) (and $q=0$)
\(\Phi_{p,0}\) is again the natural inclusion, and
$L_{p,0}=\BB_1^{p-1}(M)$.

A key observation is that the restriction of \(\Phi_{p,q}\) to
$
J_{p,q}\setminus L_{p,q}
$
is a homeomorphism onto the stratum
$
\BB_{p,q}^{\circ}(M,\partial M)
$.
Indeed, outside \(L_{p,q}\), the interior and boundary parts of the
support are uniquely determined by the resulting barycenter.
Since the strata of weighted cardinality \(l\) are precisely the
\(\BB_{p,q}^{\circ}(M,\partial M)\) with \(2p+q=l\), and since the
closures of two distinct such strata intersect only inside
\(\BB_{l-1}^\partial(M)\), we can apply Lemma \ref{pushoutmethod} and get the pushout diagram
\[
\xymatrix{
\displaystyle
\coprod_{2p+q=l} L_{p,q}
\ar[r]\ar[d]
&
\coprod_{2p+q=l} J_{p,q}\ar[d]^{\coprod\Phi_{p,q}}
\\
\BB_{l-1}^\partial(M)
\ar[r]
&
\BB_l^\partial(M).
}
\]
It remains to verify that
$
L_{p,q}\hookrightarrow J_{p,q}
$
is a cofibration. 

By Lemma \ref{lem:compatible-triangulations}, the finite barycenter spaces
\(\BB_p(M)\) and \(\BB_q(\partial M)\), and the subspaces
\[
\BB_1^{p-1}(M)\subseteq \BB_p(M),
\qquad
\BB_{q-1}(\partial M)\subseteq \BB_q(\partial M),
\]
admit compatible CW structures. Hence the two
subspaces 
$\BB_1^{p-1}(M)*\BB_q(\partial M)$
and
$\BB_p(M)*\BB_{q-1}(\partial M)$
are subcomplexes of \(J_{p,q}\), and so is their union \(L_{p,q}\).
Therefore
$L_{p,q}\hookrightarrow J_{p,q}$
is a closed cofibration.
Since a coproduct of cofibrations is a cofibration and a pushout of a
cofibration is a cofibration, the inclusion
$
\BB_{l-1}^\partial(M)
\hookrightarrow
\BB_l^\partial(M)
$
is a closed cofibration, and
$
\left(
\BB_l^\partial(M),
\BB_{l-1}^\partial(M)
\right)
$
is an NDR pair.
\end{proof} 

\begin{lem}\label{lem:contractible-in-next}
For every \(l\ge 2\), the subspace \(\BB_{l-1}^{\pa}(M)\) is contractible inside \(\BB_l^{\pa}(M)\).  Hence
$$\BB_l^\pa (M)/\BB_{l-1}^\pa (M)\simeq\BB_{l}^\pa (M)\vee \Sigma \BB_{l-1}^\pa (M)$$
In particular, with coefficients in any commutative ring
\begin{equation*}
\wt H_r(\BB_l^{\pa}(M)/\BB_{l-1}^{\pa}(M))
\cong
\wt H_r(\BB_l^{\pa}(M))
\oplus
\wt H_{r-1}(\BB_{l-1}^{\pa}(M)).
\end{equation*}
\end{lem}

\begin{proof}
It suffices to show that $\BB_{l-1}^\pa (M)$ is contractible in $\BB_l^\pa (M)$.
Choose \(y_0\in\pa M\).  For \(\sigma=\sum_i t_i\delta_{z_i}\in \BB_{l-1}^{\pa}(M)\), set
\[
H(s,\sigma)=s\delta_{y_0}+(1-s)\sigma,
\qquad 0\le s\le1.
\]
The added point lies on the boundary and costs one unit, hence \(H(s,\sigma)\in \BB_l^{\pa}(M)\).  This contracts \(\BB_{l-1}^{\pa}(M)\) to \(\delta_{y_0}\) inside \(\BB_l^{\pa}(M)\), and since the pair $(\BB_l^\pa (M), \BB_{l-1}^\pa (M))$ is an NDR pair, the quotient is up to homotopy the mapping cone on $\BB_{l-1}^\pa (M)$ with the base collapsed, thus the assertion.
\end{proof}

As in Lemma \ref{lem:contractible-in-next}, we have the splitting
\begin{equation}\label{eq:bar-B-def-general}
\BB_n(X)/\BB_{n-1}(X)\simeq \BB_n(X)\vee \Sigma \BB_{n-1}(X),\ \ n\ge 2
\end{equation}
and the homology decomposition
\begin{equation*}
\wt H_*(\BB_n(X)/\BB_{n-1}(X))
\cong
\wt H_*(\BB_n(X))
\oplus
\wt H_{*-1}(\BB_{n-1}(X))\ \ ,\ \ n\geq 2
\end{equation*}

We recall the following homotopy facts on join constructions (see for example \cite{AKN2016}).

\begin{itemize}
\item As is customary in combinatorial topology, we adopt the following augmented
join convention. Consider some formal symbol/sphere $S^{-1}$ and set
\begin{equation}\label{convention}
X*S^{-1}=S^{-1}*X=X.
\end{equation}
Whenever $\BB_0(X)$ occurs as a join factor in the formulas below, we use the augmented convention $\BB_0(X)=S^{-1}$.
\item If \((X,A)\) and \((Y,B)\) are compact CW pairs and the inclusions are cofibrations, then
\begin{equation}\label{eq:join-quotient-one-sided}
(X*Y)/(A*Y)\simeq (X/A)*Y.
\end{equation}
\item If both \(A\) and \(B\) above are nonempty, then
\begin{equation}\label{eq:join-quotient-two-sided}
(X*Y)/(A*Y\cup X*B) \simeq (X/A)*(Y/B)
\end{equation}
\item Let \(A\), \(B\), and \(C\) be well-pointed spaces having the homotopy
type of CW-complexes. Then there is a natural homotopy equivalence
\begin{equation}\label{joinwedge}
A*(B\vee C)\simeq (A*B)\vee(A*C).
\end{equation}
A quick proof of this fact uses the known identification $\displaystyle
X*Y\simeq \Sigma(X\wedge Y)$,
and the fact that the smash product \(\wedge\) 
distributes over wedges.
\end{itemize}

\subsection{The rank filtration} 
Consider the meet-semilattice diagram $\mathcal D_{2l} $ in Proposition \ref{colimitdiagram} whose
colimit is $\BB_{2l}^\pa (M)$. Filter this colimit by rank, where rank zero corresponds to the very bottom space $\mathcal D_{2l}^{(0)}=\BB_l^0(M) = \BB_l(\pa M)$, and generally $\mathcal D_{2l}^{(i)}$ is the meet-semilattice up to rank $i$ (that is, the first $i+1$ rows). Then
$$T_l^{(i)} := \colim 
\mathcal D_{2l}^{(i)} 
=
\bigcup_{0\leq j\le i}\BB_{l+i-2j}^{j}(M)$$
In particular
$
T_l^{(l)} = \colim\mathcal D_{2l}^{(l)}= \BB_{2l}^{\pa}(M),
$ and $
T_l^{(l-1)}=\colim\mathcal D_{2l}^{(l-1)}= \BB_{2l-1}^{\pa}(M).
$. 

We will now analyze the consecutive quotients of the filtration.
It is convenient to depict in Figure~\ref{fig:two-row-quotient} two consecutive rows of $\mathcal D_{2l}$. 
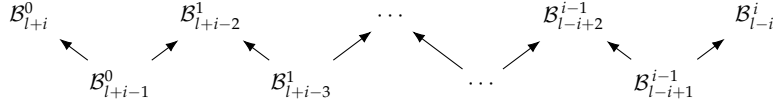
\begin{figure}[htbp]
\centering
\begin{tikzpicture}
[>=Latex,every node/.style={font=\scriptsize},x=1.2cm,y=0.9cm]

\node (a0) at (0,0) {$\BB_{l+i}^{0}$};
\node (a1) at (2,0) {$\BB_{l+i-2}^{1}$};
\node (a2) at (4,0) {$\cdots$};
\node (a3) at (6,0) {$\BB_{l-i+2}^{\,i-1}$};
\node (a4) at (8,0) {$\BB_{l-i}^{\,i}$};

\node (b0) at (1,-1) {$\BB_{l+i-1}^{0}$};
\node (b1) at (3,-1) {$\BB_{l+i-3}^{1}$};
\node (b2) at (5,-1) {$\cdots$};
\node (b3) at (7,-1) {$\BB_{l-i+1}^{\,i-1}$};

\draw[->] (b0) -- (a0);
\draw[->] (b0) -- (a1);

\draw[->] (b1) -- (a1);
\draw[->] (b1) -- (a2);

\draw[->] (b2) -- (a2);
\draw[->] (b2) -- (a3);

\draw[->] (b3) -- (a3);
\draw[->] (b3) -- (a4);

\end{tikzpicture}
\caption{Two consecutive rows in the boundary-weighted filtration.  Dividing the upper row by the lower row produces the row-quotient wedge used in \eqref{eq:Ti-quotient-general}.}
\label{fig:two-row-quotient}
\end{figure}
From Fig. \ref{fig:two-row-quotient}, it is clear that 
\begin{equation}\label{eq:Ti-quotient-general} \frac{T_l^{(i)}}{T_l^{(i-1)}}\simeq\left(\frac{\BB_{l+i}(\pa M)}{\BB_{l+i-1}(\pa M)}\right)
\vee\bigvee_{1\le j<i}
\left(
\frac{\BB_{l+i-2j}^{j}(M)}
{\BB_{l+i-2j-1}^{j}(M)\cup\BB_{l+i-2j+1}^{j-1}(M)}
\right)
\vee\left(
\frac{\BB_{l-i}^{i}(M)}{\BB_{l-i+1}^{i-1}(M)}
\right)
\end{equation}
For \(i=0\), the quotient is just \(\BB_l(\pa M)\).  The last summand is one-sided: there is no lower-row term \(\BB_{l-i-1}^{i}\) to divide out at that stage.

We have to analyze the wedge terms in this decomposition, and remarkably they can be simplified drastically. Set $A_p = \BB_p(M/\partial M)$.

\begin{lem}\label{lem:row-quotient-splitting}
For $p\geq 1$, we have the homotopy equivalence 
$\displaystyle\frac{\BB_q^p(M)}{\BB_{q+1}^{p-1}(M)} 
\simeq 
A_p*\BB_q(\pa M)$.
In particular
$$\displaystyle 
\wt H_*\left(\BB_q^p(M), \BB_{q+1}^{p-1}(M)\right)
\cong
\wt H_*(A_p*\BB_q(\pa M)).
$$
\end{lem} 

\begin{proof}
For $p\geq 1, q\geq 1$, we have the series of equivalences 
$$\frac{\BB_q^p(M)}{\BB_{q+1}^{p-1}(M)} 
\simeq
\frac{\BB_p(M)*\BB_q(\pa M)}{\BB_1^{p-1}(M)*\BB_q(\pa M)}\simeq
\frac{\BB_p(M)}{\BB_1^{p-1}(M)}*\BB_q(\pa M)
\simeq
\BB_p(M/\pa M)*\BB_q(\pa M)$$
obtained sequentially by first applying Lemma \ref{eq:Bpq-colimit}, then \eqref{eq:join-quotient-one-sided}, and finally \eqref{quotient}. For $p\ge1$ and $q=0$, the statement reduces to \eqref{quotient} using the convention \eqref{convention}.

The homology statement is evident.
\end{proof}

\begin{lem}\label{eq:two-sided-mixed-quotient}
For \(p\ge1\) and \(q\ge2\), there is a homotopy equivalence
$$\frac{\BB_q^p(M)}{\BB_{q-1}^{p}(M)\cup\BB_{q+1}^{p-1}(M)}\simeq
A_p*(\BB_q(\pa M)\vee\Sigma \BB_{q-1}(\pa M))
$$
In particular
$$\displaystyle 
\wt H_*\left(
\BB_q^p(M), \BB_{q-1}^{p}(M)\cup\BB_{q+1}^{p-1}(M)
\right) \cong
\wt H_*(A_p*\BB_q(\pa M))
\oplus
\wt H_*(A_p*\Sigma \BB_{q-1}(\pa M)).$$
\end{lem}

\begin{proof}
From the proof of Lemma \ref{ndrpair}, we observe that we have the following pushout square
\[
\xymatrix{
\left(
\BB_1^{p-1}(M)*\BB_q(\partial M)
\right)
\cup
\left(
\BB_p(M)*\BB_{q-1}(\partial M)
\right)
\ar[r]\ar[d]
&
\BB_p(M)*\BB_q(\partial M)
\ar[d]^{f}
\\
\BB_{q-1}^{p}(M)\cup\BB_{q+1}^{p-1}(M)
\ar[r]
&
\BB_q^p(M),
}
\]
from which Lemma \ref{eq:Bpq-colimit} gives the equivalence
\[
\frac{\BB_q^p(M)}
{\BB_{q-1}^{p}(M)\cup\BB_{q+1}^{p-1}(M)}
\simeq
\frac{\BB_p(M)*\BB_q(\partial M)}
{\BB_1^{p-1}(M)*\BB_q(\partial M)
\;\cup\;
\BB_p(M)*\BB_{q-1}(\partial M)}.
\]
By \eqref{eq:join-quotient-two-sided}, the latter quotient is equivalent to
\[
\frac{\BB_p(M)}{\BB_1^{p-1}(M)}
*
\frac{\BB_q(\pa M)}{\BB_{q-1}(\pa M)}.
\]
Finally, \eqref{quotient} and \eqref{eq:bar-B-def-general} identify this space with
\[
\BB_p(M/\partial M)*
(\BB_q(\pa M)\vee \Sigma\BB_{q-1}(\pa M)).
\]
The homology inference is again obvious.
\end{proof}

\begin{proof}[Proof of Theorem \ref{thm:intro-general-homology}]
We prove the claimed decomposition by induction on \(i\), for
\(0\le i\le l\).
We first rewrite \eqref{eq:Ti-quotient-general} using the homotopy equivalence in Lemma \ref{lem:row-quotient-splitting} and the decomposition in Lemma \ref{eq:two-sided-mixed-quotient}, using distributivity of the join over wedges. This gives
\begin{equation}\label{eq:Ti-quotient-expanded}
\begin{aligned}
\frac{T_l^{(i)}}
     {T_l^{(i-1)}}
\simeq {}&
\BB_{l+i}(\pa M)
\vee
\Sigma\BB_{l+i-1}(\pa M)
\\
&\vee
\bigvee_{1\leq j<i}
\left(
A_j*\BB_{l+i-2j}(\pa M)
\right)
\\
&\vee
\bigvee_{1\leq j<i}
\left(
A_j*\Sigma\BB_{l+i-2j-1}(\pa M)
\right)
\\
&\vee
\left(
A_i*\BB_{l-i}(\pa M)
\right),
\end{aligned}
\end{equation}
The null-homotopy needed at this point preserves the rank filtration.
Indeed, choose \(y_0\in\pa M\) and put
\[
H(s,\sigma)=s\delta_{y_0}+(1-s)\sigma .
\]
If \(\sigma\in\BB_{l+i-1-2j}^{j}(M)\subset T_l^{(i-1)}\), then adding
one boundary slot gives
\(H(s,\sigma)\in\BB_{l+i-2j}^{j}(M)\subset T_l^{(i)}\).
Thus \(T_l^{(i-1)}\) is null-homotopic inside \(T_l^{(i)}\). Since the
inclusion is a cofibration by Lemma \ref{lem:compatible-triangulations},
\[
\frac{T_l^{(i)}}{T_l^{(i-1)}}\simeq
T_l^{(i)}\vee\Sigma T_l^{(i-1)}.
\]
We can therefore rewrite the preceding equivalence with one suspended term
on each side:
\begin{equation}\label{eq:Ti-quotient-regrouped}
\begin{aligned}
 T_l^{(i)}\vee\Sigma T_l^{(i-1)}\simeq {}&
\left[
\BB_{l+i}(\pa M)
\vee
\bigvee_{1\leq j\leq i}
\left(
A_j*\BB_{l+i-2j}(\pa M)
\right)
\right]
\\
&\vee
\Sigma
\left[
\BB_{l+i-1}(\pa M)
\vee
\bigvee_{1\leq j<i}
\left(
A_j*\BB_{l+i-2j-1}(\pa M)
\right)
\right].
\end{aligned}
\end{equation}
If wedge cancellation were valid, induction would identify $T_l^{(i)}(M)$ with the wedge
\begin{equation}\label{splitformula}\BB_{l+i}(\pa M)
\vee
\bigvee_{1\leq j\leq i}
\left(
A_j*\BB_{l+i-2j}(\pa M)
\right)
.\end{equation}
However, cancellation for the wedge product does not hold in general: there are spaces $A,B,C$ such that $A\vee B\simeq A\vee C$, but $B\not\simeq C$. After passing to reduced homology, all groups are finite-dimensional in each degree, and the inductive isomorphism for $T_l^{(i-1)}$ permits degreewise cancellation of the common suspended summand. Thus induction on $i$ gives
\begin{equation}\label{eq:Ti-homology-final}
\wt H_*(T_l^{(i)})
\cong
\wt H_*(\BB_{l+i}(\pa M))
\oplus
\bigoplus_{j=1}^{i}
\wt H_*(\BB_j(M/\pa M)*\BB_{l+i-2j}(\pa M)).
\end{equation}

Putting \(i=l\) yields the formula
$$\wt H_*(\BB_{2l}^{\pa}(M))
\cong
\wt H_*(\BB_{2l}(\pa M))
\oplus
\wt H_*(\BB_l(M/\pa M))
\notag\ \oplus
\bigoplus_{i=1}^{l-1}
\wt H_*(\BB_i(M/\pa M)*\BB_{2l-2i}(\pa M))
$$
Here we use the convention \(\BB_0(\pa M)=S^{-1}\) and $X*S^{-1} = X$ (see \eqref{convention}), which turns the final even mixed term with $l=i=j$ into \(\BB_l(M/\pa M)\), as stated. 

We can finally use the field-coefficient join identity
\begin{equation}\label{eq:join-homology}
\wt H_*(X*Y;\F)
\cong
\Sigma\bigl(\wt H_*(X;\F)\otimes_{\F}\wt H_*(Y;\F)\bigr),
\end{equation}
to obtain the even formula \eqref{eq:intro-general-even}.
Putting \(i=l-1\) in \eqref{eq:Ti-homology-final} and then using \eqref{eq:join-homology} yields the odd formula \eqref{eq:intro-general-odd}.  
\end{proof}

We continue to write $\wt P_X(t;\F)$ for the reduced Poincar\'e
polynomial. When the field $\F$ need not be specified, we write
$\wt P_X(t)$.

\begin{cor}\label{poincarepolynomial} 
Set $M^+ = M/\pa M$. As a corollary of Theorem \ref{thm:intro-general-homology} and its proof, we have
\begin{equation*}
\begin{aligned}
\widetilde P_{\BB_{2l}^{\partial}(M)}(t)
={}&
\widetilde P_{\BB_{2l}(\pa M)}(t)
+
\widetilde P_{\BB_l(M^+)}(t)+
t\sum_{j=1}^{l-1}
\widetilde P_{\BB_j(M^+)}(t)\,
\widetilde P_{\BB_{2l-2j}(\pa M)}(t).
\end{aligned}
\end{equation*}
\begin{equation*}
\widetilde P_{\BB_{2l-1}^{\partial}(M)}(t)
=
\widetilde P_{\BB_{2l-1}(\partial M)}(t)
+
t\sum_{i=1}^{l-1}
\widetilde P_{\BB_i(M^+)}(t)\,
\widetilde P_{\BB_{2l-2i-1}(\partial M)}(t).
\end{equation*}
\end{cor}


\section{Disks and hemispheres}
\label{sec:hemispheres}

Let \(H^d=\mathbb S^d_+\) be the upper closed hemisphere, topologically homeomorphic to the closed unit \(d\)-disk.  Then \(\pa H^d=S^{d-1}\) and \(H^d/\pa H^d\simeq S^d\).
Applying Corollary 
\ref{poincarepolynomial} to $H^d$ we obtain  
\begin{equation}\label{poincareholed}
\begin{aligned}
\widetilde P_{\BB_{2l}^{\partial}(H^d)}(t)
={}&
\widetilde P_{\BB_{2l}(S^{d-1})}(t)
+
\widetilde P_{\BB_l(S^d)}(t)+
t\sum_{j=1}^{l-1}
\widetilde P_{\BB_j(S^d)}(t)\,
\widetilde P_{\BB_{2l-2j}(S^{d-1})}(t).
\end{aligned}
\end{equation}
and the analogous formula holds for the odd case $\BB_{2l-1}^\pa (H^d)$. 
We recall how the Poincar\'e polynomial for barycenters of spheres is computed.

\begin{lem}\label{poincarebarycenterspheres}
\[
\wt P_{\BB_n(S^d)}(t;\Q)=
\begin{cases}
t^{n(d+1)-1},& d\text{ odd and } n\ge 2\\
0,& d\text{ even and }n\ge2.\\
t^d,& n=1.
\end{cases}
\] 
\end{lem}

\begin{proof}
By \cite{KallelKaroui2011}, for every connected compact triangulable space \(X\), with a chosen basepoint $x_0\in X$, and for every \(n\ge1\), there is an identification
\begin{equation}\label{eq:stable-bary-symmetric}
\Sigma \BB_n(X)\simeq \overline{\SP}^{n}(\Sigma X).
\end{equation}
where $\Sigma X$ means the suspension of a space $X$, and \(\overline{\SP}^{n}(Y):=\SP^n(Y)/\SP^{n-1}(Y)\) is the reduced symmetric product. Here $\SP^n(Y)=Y^n/\mathfrak S_n$ is the quotient of $Y^n$ by the permutation action of the symmetric group on $n$-letters, and $\SP^{n-1}(Y)$ includes in $\SP^n(Y)$ by adding to the class of a tuple $[y_1,\ldots, y_{n-1}]$ the fixed point $x_0$. From \eqref{eq:stable-bary-symmetric}, it follows that
\begin{equation*}
\widetilde P_{\BB_n(S^d)}(t)
=
t^{-1}\widetilde P_{\overline{\SP}^{\,n}(S^{d+1})}(t).
\end{equation*}
The case $d=1$ can be made explicit since it is known that $\BB_n(S^1)\simeq S^{2n-1}$, while
$\SP^n(S^2)=\mathbb C P^n$ is complex projective space, so $\overline\SP^n(S^2)=S^{2n}$, and indeed
$$\wt P_{S^{2n-1}}(t) = t^{2n-1} = t^{-1}t^{2n}
=t^{-1}\wt P_{S^{2n}}(t).$$
For $d>1$, the following identities are quite classical but very well explained in Felix and Tanr\'e \cite{felixtanre}
\[
\SP^n S^{2k+1}\simeq_{\mathbb Q} S^{2k+1},
\ \ \hbox{and}\ \ 
\SP^n S^{2k}\simeq_{\mathbb Q} P^nS^{2k}
\]
where $P^nS^{2k}$ is the rational space having
$
(\Lambda a/(a^{n+1}),0)
$ as model, with $a$ of degree $2k$. (Observe that
$P^nS^2=\mathbb{C}P^n$ and
$
P^\infty S^{2k}=K(\mathbb Q,2k))
$. From this we deduce that rationally again
\[
\overline\SP^n S^{2k+1}\simeq_{\mathbb Q} *
\ \ \hbox{and}\ \ 
\overline\SP^n S^{2k}\simeq_{\mathbb Q} S^{2kn}\ \ 
\hbox{$k\geq 1$ and $n\geq 2$}
\]
and so, for $n\geq 2$
\[
\widetilde P_{\overline{\SP}^{\,n}(S^d)}(t)
=
\begin{cases}
0, & d \text{ odd},\\[2mm]
t^{dn}, & d \text{ even}.
\end{cases}
\]
This is precisely our statement after desuspending. The statement for $n=1$ is obvious since $\SP^1(X) = \overline\SP^1(X) = X$.
\end{proof}

Feeding Lemma \ref{poincarebarycenterspheres} to \eqref{poincareholed} proves immediately Corollary \ref{closeddisks}.

In particular, for the two-dimensional hemisphere \(H^2=D^2\),
\begin{align*}
\wt P_{\BB_{2s}^{\pa}(D^2)}(t;\Q)
&=
\begin{cases}t^3+t^2,&s=1,\\ t^{4s-1}+t^{4s-2},&s\ge2,
\end{cases}\notag\\
\wt P_{\BB_{2s-1}^{\pa}(D^2)}(t;\Q)
&=
\begin{cases}t,&s=1,\\ t^{4s-3}+t^{4s-4},&s\ge2,
\end{cases}
\end{align*}

The odd dimensional analog of Corollary \ref{closeddisks} is given below.

\begin{prop}\label{odddimensionalcase}
    Let \(M\) be the closed \(d\)-dimensional disk, equivalently the
hemisphere \(H^d\), and assume that \(d\ge3\) is odd. Then, with rational
coefficients,
\begin{align*}
\widetilde P_{\BB_{2l}^{\partial}(M)}(t;\mathbb Q)
&=
t^{(d+1)l-1},
\\
\widetilde P_{\BB_{2l-1}^{\partial}(M)}(t;\mathbb Q)
&=
\begin{cases}
t^{(d+1)l-2}, & l\ge2,\\[1mm]
t^{d-1}, & l=1.
\end{cases}
\end{align*}
\end{prop}

\begin{proof}
Since \(d-1\) is even, Lemma \ref{poincarebarycenterspheres} gives
\(\wt P_{\BB_n(S^{d-1})}(t;\Q)=0\) for \(n\ge2\), whereas
\(\wt P_{\BB_1(S^{d-1})}(t;\Q)=t^{d-1}\). In the even formula
\eqref{poincareholed}, every mixed term therefore vanishes, and the remaining
term is \(\wt P_{\BB_l(S^d)}(t;\Q)=t^{(d+1)l-1}\). In the odd formula,
the boundary term is nonzero only when \(l=1\). For \(l\ge2\), only the
summand with \(i=l-1\) survives, giving
\[
t\,\wt P_{\BB_{l-1}(S^d)}(t;\Q)\,
\wt P_{\BB_1(S^{d-1})}(t;\Q)
=t^{(d+1)l-2}.
\]
This proves both formulas.
\end{proof}

 
\section{Specialization to orientable surfaces}
\label{sec:surface-specialization}

Let \(\Si\) be a compact connected orientable surface of genus
\(\mathfrak g\) with \(b\geq1\) boundary components. The key quotient
identification we need is
\[
\Si^+
\simeq
\Si/\pa\Si
\simeq
S_{\mfg}\vee\bigvee^{\,b-1}S^1,
\]
where \(S_{\mfg}\) is the closed orientable surface of genus
\(\mathfrak g\). Set
$
N=2\mathfrak g+b-1,
$
and define
\begin{equation}\label{eq:surface-Q-general}
Q_{N,n}(t)
:=
\wt P_{\BB_n(\Si^+)}(t;\Q),
\qquad
R_{b,n}(t)
:=
\wt P_{\BB_n(\sqcup_{a=1}^{b}S^1)}(t;\Q).
\end{equation}

According to Corollary~\ref{poincarepolynomial}, the Poincar\'e
polynomials for the weighted barycenter spaces of \(\Si\) are immediately given by their expressions in Corollary \ref{surface-case}.
We now compute both polynomials \(Q_{N,n}(t)\) and \(R_{b,n}(t)\).

\begin{lem}\label{QNn}
If \(N\geq1\), then
\begin{equation}\label{eq:intro-Q-rational-expanded}
Q_{N,n}(t)
=
\binom{N+n-1}{n}t^{2n-1}
+
\binom{N+n-2}{n-1}t^{2n}.
\end{equation}
If \(N=0\), equivalently if \((\mathfrak g,b)=(0,1)\), then
\[
Q_{0,1}(t)=t^2,
\qquad
Q_{0,n}(t)=0
\quad\text{for }n\geq2.
\]
\end{lem}

\begin{proof}
An immediate consequence of
\eqref{eq:stable-bary-symmetric} 
\begin{equation}\label{bnsigma+}
\Sigma\BB_n(\Si^+)\simeq
\overline{\SP}^{\,n}
\left(S^3\vee\bigvee^N S^2\right),\ \ \ \ N = 2\mfg  + b-1
\end{equation}
It follows that
\begin{eqnarray*}
Q_{N,n}(t)
=
\wt P_{\BB_n(\Si^+)}(t;\Q)
=
t^{-1}
\wt P_{\overline{\SP}^{\,n}
(S^3\vee\bigvee^{N}S^2)}(t;\Q).
\end{eqnarray*}
The rational Dold--Thom calculation
\cite{Dold1995,felixtanre} shows that $H_*(\overline\SP^n(Y),\Q)$ is a subvector space of the graded symmetric algebra
\[
H_*(\SP^\infty(Y);\Q)
\cong
\operatorname{Sym}_{\mathrm{gr}}
\bigl(\wt H_*(Y;\Q)\bigr),
\]
By definition, this is  polynomial on even-degree
generators and exterior on odd-degree generators. For
$
Y=S^3\vee\bigvee^{N}S^2
$, this algebra becomes
$$
\Lambda(e_3)\otimes\Q[u_1,\ldots,u_N],
\qquad
|e_3|=3,\ 
|u_a|=2
\quad\text{(homological degree)}.
$$
To describe $H_*(\overline\SP^n(Y),\Q)$ in this algebra, we assign to every generator filtration degree one and extend this degree additively to products (i.e. $\deg(ab)=\deg a+\deg b$).
The infinite symmetric product takes wedges to products, and the filtration
is additive. 
By the theory of Dold, or by the Steenrod splitting in homology for
symmetric products, the homology
$
H_*(\overline{\SP}^{\,n}(Y);\Q)
$
corresponds precisely to the elements of filtration degree \(n\). 
These are of two types: (1) the
first type consists of products
$\displaystyle 
u_1^{i_1}\cdots u_N^{i_N},$ $
\sum_{a=1}^{N}i_a=n.$
These elements have homological degree \(2n\), and their number is
$
\binom{N+n-1}{n}
$. And (2) the second type consists of
$\displaystyle 
e_3u_1^{i_1}\cdots u_N^{i_N},$ $
\sum_{a=1}^{N}i_a=n-1
$. These elements have homological degree
$3+2(n-1)=2n+1$, 
and their number is
$\binom{N+n-2}{n-1}$.
With this count, and after desuspending once (i.e. multiplying by $t^{-1}$), we obtain
\begin{equation}\label{eq:Q-rational-surface-expanded}
Q_{N,n}(t)
=
\binom{N+n-1}{n}t^{2n-1}
+
\binom{N+n-2}{n-1}t^{2n}.
\end{equation}
This is what we wanted to prove. 
If \((\mathfrak g,b)=(0,1)\), that is if $N=0$, 
\(\Si^+\simeq S^2\), hence
$Q_{0,1}(t)=t^2$ and
$Q_{0,n}(t)=0$ for $n\geq2.$
\end{proof}

For the application in Section \ref{sec:variational-jump}, the coefficient
field is \(\mathbb F_2\).  Here the rational graded-symmetric-algebra
argument is not valid: the mod-two cohomology of
\(\SP^\infty(S^3)\simeq K(\mathbb Z,3)\) contains the additional Nakaoka
generators produced by iterated Steenrod squares. Since $H^*(\SP^\infty (X),\mathbb F_2)$ is a bigraded algebra, we define the bivariate Poincar\'e series 
$$ {\mathbb P}_{\SP^\infty (X)}(t,x) := \sum a_{ij} t^ix^j$$
where $t$ records the ordinary cohomological degree, $x$ records the symmetric-product filtration degree, and $a_{ij} = \dim_{\mathbb F_2} H^{i,j}(\SP^\infty (X);\mathbb F_2)$. We will write 
$$[x^n]\mathbb P_{\SP^\infty X}(t,x) = \sum a_{i,n}t^i$$
the polynomial term in $t$ multiplying $x^n$.

\begin{lem}\label{lem:Q-mod-two}
Define
\[
Q^{(2)}_{N,n}(t):=
\wt P_{\BB_n(\Si^+)}(t;\mathbb F_2).
\]
Then, for every \(N\ge0\) and \(n\ge1\),
\begin{equation}\label{eq:Q-mod-two}
Q^{(2)}_{N,n}(t)
=
t^{-1}[x^n]
\frac{1}{(1-xt^2)^N}
\prod_{r\ge0}
\frac{1}{1-x^{2^r}t^{2^{r+1}+1}}.
\end{equation}
Only the factors with \(2^r\le n\) contribute to \([x^n]\), so
\eqref{eq:Q-mod-two} is a finite, explicit coefficient formula for each
fixed \(n\).  In low orders it gives
\begin{equation}\label{eq:Q-mod-two-low-orders}
Q^{(2)}_{N,1}(t)=Nt+t^2,
\qquad
Q^{(2)}_{N,2}(t)
=\binom{N+1}{2}t^3+(N+1)t^4+t^5.
\end{equation}
In particular,
\(Q^{(2)}_{0,2}(t)=t^4+t^5\).
\end{lem}

\begin{proof}
The starting point is again \eqref{bnsigma+}.
As already mentioned, the Steenrod splitting identifies the cohomology of the reduced symmetric
product with the part of exact filtration \(n\) in the cohomology of the
infinite symmetric product.  The \(N\) copies of \(S^2\) contribute
polynomial generators \(u_1,\ldots,u_N\), each of bidegree
\((2,1)\), where the second entry is the symmetric-product filtration degree.
For the \(S^3\)-summand, it is well-known that
$\SP^\infty (S^3)$ is a model for a $K(\mathbb Z,3)$ of which mod $2$ cohomology is given below
\[
H^*(\SP^\infty(S^3);\mathbb F_2)
\cong
\mathbb F_2[z_0,z_1,z_2,\ldots]
\]
Here \(z_0=\iota_3\) is the dual class to the embedded sphere $S^3=\SP^1(S^3)\hookrightarrow \SP^\infty(S^3)$, and for \(r\ge1\),
\(z_r=\operatorname{Sq}^{2^r}\cdots
\operatorname{Sq}^{4}\operatorname{Sq}^{2}(\iota_3)\), where the $Sq^i$ are the Steenrod operations \cite{serre}.
The cohomological degree of $z_r$ is $2^{r+1}+1$. The filtration degree of $z_r$ is $2^r$ 
(see \cite[Theorem~11.2 and Example~11.3]{KallelKaroui2011}).
The bivariate Poincar\'e series of the resulting
polynomial algebra is computed as follows
\[
{\mathbb P}_{\SP^\infty (\Sigma\Sigma^+)}(t,x) = 
\frac{1}{(1-xt^2)^N}
\prod_{r\ge0}\frac{1}{1-x^{2^r}t^{2^{r+1}+1}}.
\]
Since $\wt P_{\BB_n(\Si^+)}(t;\mathbb F_2) = \frac{1}{t}[x^n]\mathbb P_{\SP^\infty (\Sigma\Sigma^+)}(t,x)$, we obtain
\eqref{eq:Q-mod-two}.  Extracting the coefficients of \(x\) and \(x^2\)
gives \eqref{eq:Q-mod-two-low-orders}.
\end{proof}

 The next term we need to determine is $R_{b,n}(t)$ in \eqref{eq:surface-Q-general}.  To that end, we recall the general two-component decomposition obtained in \cite{AKN2016}. As in \eqref{convention}, we have that $\BB_0(X)*Y=Y*\BB_0(X) = Y$.

\begin{thm}\label{thm:disjoint-two-components}
Let \(A\) and \(B\) be two non-empty spaces of the homotopy type of a finite CW-complex. For \(n\ge2\), and with field coefficients,
\begin{align*}
\wt H_*\bigl(\BB_n(A\sqcup B)\bigr)
\cong {}&
\bigoplus_{X\in\{A,B\}}
\left[
\wt H_*\bigl(\BB_n(X)\bigr)
\oplus
\wt H_{*-1}\bigl(\BB_{n-1}(X)\bigr)
\right]
\notag\\
&\oplus
\bigoplus_{\substack{p,q\ge1\\ p+q=n}}
\wt H_*\bigl(\BB_p(A)*\BB_q(B)\bigr)
\notag\\
&\oplus
\bigoplus_{\substack{p,q\ge1\\ p+q=n-1}}
\wt H_{*-1}\bigl(\BB_p(A)*\BB_q(B)\bigr).
\end{align*}
\end{thm}

\begin{proof} This is a restatement, after reindexing, of the decomposition obtained in \cite[Theorem~5.19]{AKN2016}. In that reference, the theorem is stated for connected $A$ and $B$, but its proof filters by the clopen partition of the support between \(A\) and \(B\) and then uses only cofibration and pushout identifications; no step uses paths within either component. Thus nonemptiness and the CW-type hypotheses suffice. The restriction \(n\ge2\) is retained because the augmented convention \(\BB_0=S^{-1}\) used in that decomposition does not encode the extra reduced \(H_0\)-class of a disjoint union at \(n=1\).
\end{proof}
 
\begin{exa}\rm The order \(n=1\) is handled separately:
\[
\widetilde H_*(\BB_1(A\sqcup B))
=\widetilde H_*(A\sqcup B)
\cong\widetilde H_*(A)\oplus\widetilde H_*(B)\oplus\F[0],
\]
where \(\F[0]\) is the one-dimensional class detecting the two connected
components. For \(n=2\), Theorem \ref{thm:disjoint-two-components} gives
    \[
\begin{aligned}
\widetilde H_*\bigl(\BB_2(A\sqcup B)\bigr)
\cong {}
\widetilde H_*\bigl(\BB_2(A)\bigr)
\oplus
\widetilde H_{*-1}(A)
\oplus
\widetilde H_*\bigl(\BB_2(B)\bigr)
\oplus
\widetilde H_{*-1}(B)
\oplus
\widetilde H_*\bigl(A*B\bigr).
\end{aligned}
\]
\end{exa}

Theorem \ref{thm:disjoint-two-components} is now used to derive the following closed formula.

\begin{prop}\label{prop:Rbn-formula}
For \(b\ge1\) and \(n\ge1\), set
$\displaystyle
R_{b,n}(t):=
\wt P_{\BB_n(\sqcup_{a=1}^{b}S^1)}(t).
$
Then
\begin{equation}\label{eq:Rbn-explicit}
R_{b,n}(t)=
\sum_{c=0}^{\min\{b-1,n\}}
\binom{b-1}{c}\binom{b+n-c-1}{n-c}t^{2n-c-1}.
\end{equation}
\end{prop}

\begin{proof}
For \(n=1\), direct reduced homology gives
\[
R_{b,1}(t)=(b-1)+bt,
\]
which is \eqref{eq:Rbn-explicit}. The case \(b=1\) is
\[
R_{1,n}(t)=\wt P_{\BB_n(S^1)}(t)=\wt P_{S^{2n-1}}(t)=t^{2n-1}
\]
over every field. Now let \(b\ge2\) and \(n\ge2\), suppose the formula is
known for \(b-1\) components, and write
\(X_{b-1}=\sqcup_{a=1}^{b-1}S^1\). Apply Theorem
\ref{thm:disjoint-two-components} to \(X_{b-1}\sqcup S^1\):
\begin{itemize}
    \item Use the identity
$\displaystyle\wt P_{X*Y}(t)=t\wt P_X(t)\wt P_Y(t)$. 
\item Summing the terms in Theorem \ref{thm:disjoint-two-components} yields
\[
R_{b,n}(t)
=
R_{b-1,n}(t)
+t^{2n-1}+t^{2n-2}
+
\sum_{\substack{p,q\geq1\\p+q=n}}
\bigl(t^{2q}+t^{2q-1}\bigr)R_{b-1,p}(t).
\]
\item Write the generating series
\[
F_b(x,t)
:=
1+t\sum_{n\geq1}R_{b,n}(t)x^n.
\]
Together with \(R_{b,1}=R_{b-1,1}+1+t\), the recurrence in the previous
bullet gives
\begin{equation}\label{iteration}
F_b(x,t)
=
\frac{1+tx}{1-t^2x}\,F_{b-1}(x,t).
\end{equation}
\item For \(b=1\),
$R_{1,n}(t)=t^{2n-1}$, 
and therefore
$$ 
F_1(x,t)
=
1+t\sum_{n\geq1}t^{2n-1}x^n
=
1+\sum_{n\geq1}(t^2x)^n
=
\frac{1}{1-t^2x}.
$$
\item Iterating \eqref{iteration} gives the closed formula for the generating series
$\displaystyle 
F_b(x,t)
=
\frac{(1+tx)^{b-1}}{(1-t^2x)^b}.
$
\item The coefficient of $x^n$ in the generating series is $tR_{b,n}(t)$.  This is precisely \eqref{eq:Rbn-explicit} and is extracted from the negative binomial expansion
\[
(1-t^2x)^{-b}
=
\sum_{m\geq0}
\binom{b+m-1}{m}t^{2m}x^m,
\]
and the standard combinatorial formula
$\displaystyle 
(1+tx)^{b-1}
=
\sum_{c=0}^{b-1}
\binom{b-1}{c}t^c x^c
$. 
\end{itemize}
\end{proof}

\begin{rem}
    \rm The proof shows that $R_{b,n}(t)$ is independent of the coefficient field. This is not the case for $Q_{N,n}(t)$.
\end{rem}

Combining Corollary \ref{poincarepolynomial}, Proposition
\ref{prop:Rbn-formula}, and Lemma \ref{lem:Q-mod-two} gives the coefficient
version needed below:
\begin{equation}\label{eq:surface-even-mod-two}
\wt P_{\BB_{2s}^{\pa}(\Si)}(t;\mathbb F_2)
=
R_{b,2s}(t)+Q^{(2)}_{N,s}(t)
+t\sum_{i=1}^{s-1}Q^{(2)}_{N,i}(t)R_{b,2s-2i}(t),
\end{equation}
and
\begin{equation}\label{eq:surface-odd-mod-two}
\wt P_{\BB_{2s-1}^{\pa}(\Si)}(t;\mathbb F_2)
=
R_{b,2s-1}(t)
+t\sum_{i=1}^{s-1}Q^{(2)}_{N,i}(t)R_{b,2s-2i-1}(t).
\end{equation}

\begin{exa}\label{exa:disk}\rm We discuss simple cases of $\BB_l^\pa (\Sigma)$.
\begin{itemize}
\item For \(n=1\), it gives
\begin{equation}\label{eq:Rb1}
R_{b,1}(t)=(b-1)+bt,
\end{equation}
which records correctly that \(\BB_1(\sqcup_bS^1)\) is the disjoint union of \(b\) circles.
\item For the disk \(D^2\), $\mfg =0$ and \(b=1\), \(N=0\).  Thus \(R_{1,n}(t)=t^{2n-1}\), \(Q_{0,1}=t^2\), and \(Q_{0,i}=0\) for \(i\ge2\).  Hence
\[
\wt P_{\BB_2^{\pa}(D^2)}(t;\Q)=t^3+t^2,
\]
which agrees with the direct computation in Remark \ref{rem:not-homotopy-invariant}.  For \(l\ge2\), the closed disk formulas are
\[
\wt P_{\BB_{2l}^{\pa}(D^2)}(t;\Q)=t^{4l-1}+t^{4l-2},
\qquad
\wt P_{\BB_{2l-1}^{\pa}(D^2)}(t;\Q)=t^{4l-3}+t^{4l-4}.
\]
\item 
For an annulus, \(\mathfrak g=0\), \(b=2\), hence \(N=1\).  Formula \eqref{eq:Q-rational-surface-expanded} gives
\[
Q_{1,n}(t)=t^{2n-1}+t^{2n}.
\]
Together with \(R_{2,n}(t)=(n+1)t^{2n-1}+n t^{2n-2}\), this gives a completely explicit rational Betti polynomial for every \(\BB_l^{\pa}\)(annulus).
\end{itemize}
\end{exa}

\begin{rem}\rm
    It is a fun exercise in combinatorics to reproduce the following formulas for the Euler characteristic given in Theorem \ref{thm:intro-euler}
\[
\chi\bigl(\BB^\partial_{2l}(\Sigma)\bigr)
=
1-\binom{N+l-1}{l},
\qquad
\chi\bigl(\BB^\partial_{2l-1}(\Sigma)\bigr)
=
1-\binom{N+l-2}{l-1}.
\]
The combinatorics are based on the Lemmas above that enter into the expression of $\wt P_{\BB_{2l}^{\pa}(\Si)}(t;\Q)$ and
$\wt P_{\BB_{2l-1}^{\pa}(\Si)}(t;\Q)$ in Corollary \ref{surface-case} (i.e. substitute \(t=-1\), $N= 2\mathfrak g + b-1$ and
$\chi (\Sigma )=2-2\mathfrak g-b=1-N$ in all formulas).
\end{rem}


\section{Boundary-weighted barycenters and the mean-field equation on surfaces with boundary}
\label{sec:variational-jump}

We specialize the topological computation to the resonant Neumann mean-field
equation in the form proved by Ahmedou--Hu--Wang-Zhang \cite{AHW2026}. Let
\((\Si,g)\) be a compact connected orientable Riemannian surface with
nonempty smooth boundary and let \(V\in C^4(\Si)\) be positive. On
\[
\mathcal H:=\left\{u\in H^1(\Si):\int_\Si u\,dv_g=0\right\},
\]
endowed with the Dirichlet norm
\(\|u\|_{\mathcal H}^2:=\int_\Si |\nabla_g u|^2\,dv_g\),
set
\begin{equation}\label{eq:Neumann-MF-functional}
J_\rho(u):=\frac12\int_\Si |\nabla_g u|^2\,dv_g
-\rho\log\int_\Si V e^u\,dv_g .
\end{equation}
Its critical points solve
\begin{equation}\label{eq:Neumann-MF-equation}
\begin{cases}
-\Delta_g u=\rho\left(\dfrac{V e^u}{\int_\Si V e^u\,dv_g}
-\dfrac1{|\Si|_g}\right)&\text{in }\Si,\\[1.1em]
\partial_{\nu_g}u=0&\text{on }\pa\Si,\\[0.3em]
\displaystyle\int_\Si u\,dv_g=0.
\end{cases}
\end{equation}
An interior bubble has mass \(8\pi\), while a boundary bubble has mass
\(4\pi\). We therefore fix
\begin{equation}\label{eq:rho-kappa-resonance}
\rho=4\pi\kappa,\qquad \kappa\ge2,
\end{equation}
so a mixed configuration of type \((p,q)\) has weighted cost
$
2p+q.
$
The order \(\kappa=1\) is excluded here: \(J_{4\pi}\) is bounded below, so
the nonempty very negative sublevel used below is not available.

We first state precisely the finite-dimensional assumptions under which the
analytic topology theorem is used. For \(2p+q=\kappa\), let
\[
\Xi_{p,q}:=
\left\{(a_1,\ldots,a_p,b_1,\ldots,b_q):
a_i\in\mathring\Si,\ b_j\in\pa\Si,\ \text{all points distinct}\right\},
\]
and let
\(\mathfrak C_{p,q}(\Si)=\Xi_{p,q}/(\mathfrak S_p\times\mathfrak S_q)\).
The reduced Kirchhoff--Routh functional
\(\mathcal F_{p,q}^V\) and the common-scale coefficient are those of
\cite[(1.3)--(1.6)]{AHW2026}. In particular,
\begin{equation}\label{eq:L-piecewise}
\mathcal L_{p,q}^V(\xi)=
\begin{cases}
\mathcal L_1^V(\xi),&q>0,\\
\mathcal L_2^V(\xi),&q=0,
\end{cases}
\end{equation}
where \(\mathcal L_1^V\) is the mixed/pure-boundary coefficient of order
\(\lambda^{-1}\), whereas \(\mathcal L_2^V\) is the pure-interior
coefficient of order \(\lambda^{-2}\ln\lambda\).
For every type \(2p+q=\kappa\), assume:
\begin{enumerate}[(C1)]
\item \(\mathcal F_{p,q}^V\) is Morse on \(\Xi_{p,q}\) and has only
finitely many critical orbits under
\(\mathfrak S_p\times\mathfrak S_q\);
\item \(\mathcal L_{p,q}^V(\xi)\ne0\) at every critical point
\(\xi\) of \(\mathcal F_{p,q}^V\).
\end{enumerate}
These are exactly conditions (C1)--(C2) of \cite{AHW2026}; in particular,
they are imposed on all reduced critical configurations, not only on the
escaping ones.

For \(a\in\R\), write
\(J_\rho^a:=\{u\in\mathcal H:J_\rho(u)\le a\}\). Under (C1)--(C2),
Ahmedou--Hu--Wang prove that, for all sufficiently large \(A\),
\begin{equation}\label{eq:MF-low-sublevel-model}
J_{4\pi\kappa}^{-A}\simeq \BB_{\kappa-1}^{\pa}(\Si).
\end{equation}
The resonance-specific step in this assertion is essential. Choose
\(\tau>0\) so that
\[
4\pi(\kappa-1)<\frac{4\pi\kappa}{1+\tau}<4\pi\kappa
\]
and, for \(L\gg1\), set
\[
U_0(L):=J_{4\pi\kappa}^{-4L},\qquad
U_1(L):=
\left\{J_{4\pi\kappa}(u)+\frac{\tau}{2}\|u\|_{\mathcal H}^2
\le-4L\right\}.
\]
The parameterized Bahri--Lucia deformation retracts \(U_0(L)\) onto
\(U_1(L)\) without crossing a genuine critical value or a critical value at
infinity. On the double of \(\Si\), the concentration projection and the
bubble test map are constructed equivariantly. The fixed barycenters under
reflection are precisely the configurations in
\(\BB_{\kappa-1}^{\pa}(\Si)\): an interior point occurs with its reflected
partner and therefore has cost two, while a boundary point is fixed and has
cost one. Restricting the equivariant homotopies to the fixed set yields
\eqref{eq:MF-low-sublevel-model}; see
\cite[Lemma~4.2 and Proposition~4.2]{AHW2026}. Thus the appearance of
\(\kappa-1\) is a consequence of the reinforced endpoint construction, not
of a direct nonresonant projection argument.

After increasing \(A\), the upper sublevel \(J_{4\pi\kappa}^{A}\) is
contractible. Ahmedou--Hu--Wang formulate the resulting Morse theory over
\(\mathbb F_2\); hence, for \(r\ge1\),
\begin{equation}\label{eq:MF-relative-pair}
H_r(J_{4\pi\kappa}^{A},J_{4\pi\kappa}^{-A};\mathbb F_2)
\cong
\wt H_{r-1}(\BB_{\kappa-1}^{\pa}(\Si);\mathbb F_2),
\end{equation}
and the relative \(H_0\)-group vanishes. Define
\begin{equation}\label{eq:MF-pair-polynomial}
P_\kappa(t):=
\sum_{r\ge0}\dim_{\mathbb F_2}
H_r(J_{4\pi\kappa}^{A},J_{4\pi\kappa}^{-A};\mathbb F_2)t^r
=t\,\wt P_{\BB_{\kappa-1}^{\pa}(\Si)}(t;\mathbb F_2).
\end{equation}
Equations \eqref{eq:surface-even-mod-two}--\eqref{eq:surface-odd-mod-two},
together with the filtered coefficient formula \eqref{eq:Q-mod-two} and
\eqref{eq:Rbn-explicit}, give every coefficient of \(P_\kappa(t)\)
explicitly. The rational formulas in Section
\ref{sec:surface-specialization} remain valid rational computations, but
they are not substituted into this mod-two Morse relation.

Set
\[
\widetilde c_j^{\kappa-1}:=
\dim_{\mathbb F_2}\widetilde H_j
\bigl(\BB_{\kappa-1}^{\pa}(\Si);\mathbb F_2\bigr).
\]
Then
\begin{equation}\label{eq:MF-pair-coefficients}
P_\kappa(t)=\sum_{r\ge1}\widetilde c_{r-1}^{\kappa-1}t^r.
\end{equation}

The sign convention in \cite{AHW2026} is that the decreasing pseudogradient
drives the common scale to \(+\infty\) exactly when
\(\mathcal L_{p,q}^V(\xi)<0\). Accordingly, the critical points at infinity
of type \((p,q)\) are precisely
\begin{equation}\label{eq:Vminus-pq-def}
\mathcal V_{p,q}^-:=
\left\{[\xi]\in\mathfrak C_{p,q}(\Si):
\nabla\mathcal F_{p,q}^V(\xi)=0,\ \mathcal L_{p,q}^V(\xi)<0\right\}.
\end{equation}
Thus the whole cost-\(\kappa\) configuration space is only the potential end
stratum; the points in \(\mathcal V_{p,q}^-\) are the actual escaping ends.
Their index at infinity is
\begin{equation}\label{eq:MF-index-at-infinity}
\iota_\infty(\xi)=
3p+2q-1-\operatorname{morse}(\mathcal F_{p,q}^V,\xi).
\end{equation}
Indeed, \(2p+q\) location variables and \(p+q-1\) relative weight/scale
variables give the barycentric dimension \(3p+2q-1\).

Assume in addition that all critical points of \(J_{4\pi\kappa}\) are
nondegenerate. The compactness theorem at the resonant level then implies
that \(\operatorname{Crit}(J_{4\pi\kappa})\) is finite. The strong Morse relation of
\cite[Theorem~1.3]{AHW2026} is
\begin{equation}\label{eq:MF-strong-Morse-final}
\sum_{u\in\operatorname{Crit}(J_{4\pi\kappa})}
t^{\operatorname{morse}(J_{4\pi\kappa},u)}
+
\sum_{2p+q=\kappa}\ \sum_{\xi\in\mathcal V_{p,q}^-}
t^{3p+2q-1-\operatorname{morse}(\mathcal F_{p,q}^V,\xi)}
=P_\kappa(t)+(1+t)\mathcal Q(t),
\end{equation}
where \(\mathcal Q(t)\) has nonnegative integer coefficients. If
\[
\nu_i:=\#\{u\in\operatorname{Crit}(J_{4\pi\kappa}):
\operatorname{morse}(J_{4\pi\kappa},u)=i\}
\]
and
\[
m_i^\infty:=
\#\left\{\xi\in\bigcup_{2p+q=\kappa}\mathcal V_{p,q}^-:
\iota_\infty(\xi)=i\right\},
\]
then equivalently
\begin{equation}\label{eq:MF-coefficient-Morse-relation}
\sum_i(\nu_i+m_i^\infty)t^i
=\sum_{r\ge1}\widetilde c_{r-1}^{\kappa-1}t^r
+(1+t)\mathcal Q(t).
\end{equation}
Evaluating at \(t=-1\) gives
\begin{equation}\label{eq:MF-euler-hopf-final}
\sum_j(-1)^j\nu_j+
\sum_{2p+q=\kappa}\ \sum_{\xi\in\mathcal V_{p,q}^-}
(-1)^{3p+2q-1-\operatorname{morse}(\mathcal F_{p,q}^V,\xi)}
=1-\chi(\BB_{\kappa-1}^{\pa}(\Si)).
\end{equation}
Equations \eqref{eq:MF-strong-Morse-final} and
\eqref{eq:MF-euler-hopf-final} are therefore coefficient-consistent,
respectively mod-two and Euler-characteristic forms of the same topology
jump.

\begingroup
\footnotesize
\medskip
\noindent
\textsc{Mohameden Ahmedou}\par
Mathematisches Institut, Justus-Liebig-Universit\"at Giessen,
Arndtstrasse 2, 35392 Giessen, Germany.\par
\texttt{Mohameden.Ahmedou@math.uni-giessen.de}

\smallskip
\noindent
\textsc{Sadok Kallel}\par
American University of Sharjah, United Arab Emirates, and
Laboratoire Paul Painlev\'e, Universit\'e de Lille, France.\par
\texttt{sadok.kallel@univ-lille.fr}
\endgroup

\end{document}